\pdfoutput=1
\RequirePackage{ifpdf}
\ifpdf 
\documentclass[pdftex]{sigma}
\else
\documentclass{sigma}
\fi

\usepackage{subfigure,enumerate}

\numberwithin{equation}{section}

\DeclareMathOperator{\Ai}{Ai}

\DeclareMathOperator{\J}{J}

\newcommand{\Ja}{\J_\alpha}

\newcommand{\goto}{\rightarrow}
\newcommand{\bigo}{{\mathcal O}}

\newcommand{\half}{\frac{1}{2}}

\DeclareMathOperator{\diag}{diag}

\newcommand{\D}{\mathrm{d}}
\newcommand{\I}{\mathrm{i}}
\newcommand{\E}{\mathrm{e}}
\newcommand{\La}[1]{L_{#1}^{(\alpha)}}

\newtheorem{Theorem}{Theorem}[section]
\newtheorem{Lemma}[Theorem]{Lemma}
\newtheorem{Proposition}[Theorem]{Proposition}
\newtheorem{Condition}[Theorem]{Condition}
 { \theoremstyle{definition}
\newtheorem{Remark}[Theorem]{Remark} }

\begin{document}


\renewcommand{\thefootnote}{$\star$}

\newcommand{\arXivNumber}{1605.06438}

\renewcommand{\PaperNumber}{109}

\FirstPageHeading

\ShortArticleName{Smoothed Analysis for the Conjugate Gradient Algorithm}

\ArticleName{Smoothed Analysis\\ for the Conjugate Gradient Algorithm\footnote{This paper is a~contribution to the Special Issue on Asymptotics and Universality in Random Matrices, Random Growth Processes, Integrable Systems and Statistical Physics in honor of Percy Deift and Craig Tracy. The full collection is available at \href{http://www.emis.de/journals/SIGMA/Deift-Tracy.html}{http://www.emis.de/journals/SIGMA/Deift-Tracy.html}}}

\Author{Govind MENON~$^\dag$ and Thomas TROGDON~$^\ddag$}

\AuthorNameForHeading{G.~Menon and T.~Trogdon}

\Address{$^\dag$~Division of Applied Mathematics, Brown University,\\
\hphantom{$^\dag$}~182 George St., Providence, RI 02912, USA}
\EmailD{\href{mailto:govind_menon@brown.edu}{govind\_menon@brown.edu}}

\Address{$^\ddag$~Department of Mathematics, University of California,\\
\hphantom{$^\ddag$}~Irvine, Rowland Hall, Irvine, CA, 92697-3875, USA}
\EmailD{\href{mailto:ttrogdon@math.uci.edu}{ttrogdon@math.uci.edu}}

\ArticleDates{Received May 23, 2016, in f\/inal form October 31, 2016; Published online November 06, 2016}

\Abstract{The purpose of this paper is to establish bounds on the rate of convergence of the conjugate gradient algorithm when the underlying matrix is a random positive def\/inite perturbation of a deterministic positive def\/inite matrix. We estimate all f\/inite moments of a natural halting time when the random perturbation is drawn from the Laguerre unitary ensemble in a critical scaling regime explored in Deift et~al.~(2016). These estimates are used to analyze the expected iteration count in the framework of smoothed analysis, introduced by Spielman and Teng (2001). The rigorous results are compared with numerical calculations in several cases of interest.}


\Classification{60B20; 65C50; 35Q15}

\medskip

\rightline{\it In honor of Percy Deift and Craig Tracy on their 70th birthdays}

\renewcommand{\thefootnote}{\arabic{footnote}}
\setcounter{footnote}{0}

\section{Introduction}
It is conventional in numerical analysis to study the worst-case behavior of algorithms, though it is often the case that worst-case behavior is far from typical. A fundamental example of this nature is the behavior of LU factorization with partial pivoting. While the worst-case behavior of the growth factor in LU factorization with partial pivoting is exponential, the algorithm works much better in practice on `typical' problems. The notion of {\em smoothed analysis} was introduced by Spielman and his co-workers to distinguish between the typical-case and worst-case performance for numerical algorithms in such situations (see~\cite{Sankar2006, Spielman2001}). In recent work, also motivated by the distinction between typical and worst case behavior, the authors (along with P.~Deift, S.~Olver and C.~Pfrang) investigated the behavior of several numerical algorithms with random input~\cite{Deift2014, DiagonalRMT} (see also~\cite{Sagun2015}). These papers dif\/fer from smoothed analysis in the sense that `typical performance' was investigated by viewing the algorithms as dynamical systems acting on random input. The main observation in our numerical experiments was an empirical universality of f\/luctuations for halting times. Rigorous results on universality have now been established in two cases~-- the conjugate gradient algorithm and certain eigenvalue algorithms~\cite{Deift2016, Deift2015}. The work of~\cite{Deift2016} presents a true universality theorem. The work~\cite{Deift2015} revealed the unexpected emergence of Tracy--Widom f\/luctuations around the {\em smallest} eigenvalue of LUE matrices in a regime where universality emerges. Thus, the analysis of a question in probabilistic numerical analysis led to a new discovery in random matrix theory.

Our purpose in this paper is to explore the connections between our work and smoothed analysis. We show that the results of~\cite{Deift2015} extend to a smoothed analysis of the conjugate gradient algorithm over \emph{strictly} positive def\/inite random perturbations (which constitute a natural class of perturbations for the conjugate gradient algorithm). To the best of our knowledge, this is the f\/irst instance of smoothed analysis for the conjugate gradient algorithm. More precisely, the results of~\cite{Deift2015} are used to establish rigorous bounds on the expected value of a halting time (Theorem~\ref{t:smooth} below). These bounds are combined with numerical experiments that show an interesting improvement of the conjugate gradient algorithm when it is subjected to random perturbations. In what follows, we brief\/ly review the conjugate gradient algorithm, smoothed analysis and the Laguerre unitary ensemble, before stating the main theorem, and illustrating it with numerical experiments.

\subsection{The conjugate gradient algorithm}
The conjugate gradient algorithm is a Krylov subspace method to solve the linear system $Ax=b$ when $A >0$ is a positive def\/inite matrix. In this article, we focus on Hermitian positive def\/inite matrices acting on $\mathbb{C}^N$, though the ideas extend to real, symmetric positive def\/inite matrices.
We use the $\ell^2$ inner product on $\mathbb{C}^N$, $ \langle u,v \rangle_{\ell^2} = \sum\limits_{i=1}^N u_i \bar{v}_i$, and $A >0$ means that $A$ is Hermitian and $ \langle u, Au \rangle_{\ell^2} >0$ for all $u \neq 0$. When $A>0$, its inverse $A^{-1}>0$, and we may def\/ine the norms
\begin{gather*}
\|u\|^2_{w} = \langle u, Au \rangle_{\ell^2} \qquad \mathrm{and} \qquad\|u\|^2_{w^{-1}} = \big\langle u, A^{-1}u \big\rangle_{\ell^2}.
\end{gather*}
In this setting, the simplest formulation of the conjugate gradient algorithm is as follows~\cite{Greenbaum1989, Hestenes1952}. In order to solve $Ax=b$, we def\/ine the increasing sequence of Krylov subspaces
\begin{gather*}
\mathcal K_k = \operatorname{span}\big\{b,Ab,\ldots,A^{k-1}b\big\},
\end{gather*}
and choose the iterates $\{x_k\}_{k=1}^\infty$, { $x_0 = 0$}, to minimize the residual $r_k = b-Ax_k$ in the $w^{-1}$ norm:
\begin{gather*}
x_k = \mathop{\operatorname{argmin}}\limits_{x \in \mathcal K_k}	\|Ax-b\|_{w^{-1}}, \qquad \|Ax_k-b\|_{w^{-1}} = \min_{x \in \mathcal K_k} \|Ax-b\|_{w^{-1}}.
\end{gather*}
Since $A>0$, $x \in \mathcal{K}_k$ for some $k \leq N$ (in our random setting it follows that $\mathcal K_N = \mathbb{C}^N$ with probability~1), so that the method takes at most $N$ steps in exact artithmetic\footnote{In calculations with f\/inite-precision arithmetic the number of steps can be much larger than $N$ and this will be taken into account in the numerical experiments in Section~\ref{s:numerics}. The results presented here can be extended to the f\/inite-precision case but only in the limit as the precision tends to $\infty$ using \cite{Greenbaum1989}. Tightening these estimates remains an important open problem.}. However, the residual decays exponentially fast, and a useful approximation is obtained in much fewer than $N$ steps. Let $\lambda_{\max}$ and $\lambda_{\min}$ denote the largest and smallest eigenvalues of $A$ and $\kappa = \lambda_{\max}/\lambda_{\min}$ the condition number. Then the rate of convergence in the $\ell^2$ and $w^{-1}$ norms is~\cite[Theorem~10.2.6]{Golub2013}
\begin{gather}\label{rut}
\|r_k\|_{w^{-1}} \leq 2\left( \frac{\sqrt{\kappa}-1}{\sqrt{\kappa}+1}\right)^{k} \|r_0\|_{w^{-1}}.
\end{gather}
Since $A$ is positive def\/inite, we have
\begin{gather*}
\lambda_{\max}^{-1} \|x\|_{\ell^2}^2 \leq \big\langle x, A^{-1} x \big\rangle_{\ell^2} \leq \lambda_{\min}^{-1} \|x\|_{\ell^2}^2.
\end{gather*}
Applying this estimate to (\ref{rut}) we f\/ind the rate of convergence of the residual in the $\ell^2$ norm
\begin{gather}\label{rutA}
\|r_k\|_{\ell^2} \leq 2\sqrt{\kappa}\left( \frac{\sqrt{\kappa}-1}{\sqrt{\kappa}+1}\right)^{k} \|r_0\|_{\ell^2}.
\end{gather}

These rates of convergence provide upper bounds on the following $\epsilon$-dependent run times, which we call \emph{halting times}:
\begin{gather}
\tau_\epsilon(A,b) = \min\left\{k\colon \frac{\|r_{k+1}\|_{\ell^2}}{\|r_0\|_{\ell^2}} \leq \epsilon\right\},\nonumber \\
\tau_{w,\epsilon}(A,b) = \min\left\{k\colon \frac{\|r_{k+1}\|_{w^{-1}}}{\|r_0\|_{w^{-1}}} \leq \epsilon \right\}.\label{haltA}
\end{gather}
Note that we have set $x_0=0$ so that $r_0 = b$ (the estimates above hold for arbitrary $x_0$). In what follows, we will also assume that $\|b\|_{\ell^2} = 1$, so that the def\/initions above simplify further.

\subsection{Smoothed analysis}
Our main results are a theorem (Theorem~\ref{t:smooth}) along with numerical evidence to demonstrate that the above worst-case estimates, can be used to obtain bounds on average-case behavior in the sense of smoothed analysis. In order to state the main result, we f\/irst review two basic examples of smoothed analysis~\cite{Spielman2001}, since these examples clarify the context of our work.

Roughly speaking, the smoothed analysis of a deterministic algorithm proceeds as follows. Given a deterministic problem, we perturb it randomly, compute the expectation of the run-time for the randomly perturbed problem and then take the maximum over all deterministic problems within a f\/ixed class. Subjecting a deterministic problem to random perturbations provides a~realistic model of `typical performance', and by taking the maximum over all deterministic problems within a natural class, we retain an important aspect of worst-case analysis. A~para\-me\-ter~$\sigma^2$ (the variance in our examples) controls the magnitude of the random perturbation. The f\/inal estimate of averaged run-time should depend explicitly on~$\sigma^2$ in way that demonstrates that the average run-time is much better than the worst-case. Let us illustrate this idea with examples.

\subsubsection{Smoothed analysis: The simplex algorithm}
Assume $\bar A = (\bar a_1, \bar a_2, \ldots, \bar a_d)$ is a deterministic matrix of size $N\times d$, and $\bar y$ and $z$ are deterministic vectors of size $N$ and $d$, respectively. Let $T(\bar A,\bar y,z)$ be the number of simplex steps required to solve the linear program
\begin{gather*}
\text{maximize} \ \ z^T x,\quad \text{subject to} \quad \bar Ax \leq \bar y,
\end{gather*}
with the two-phase shadow-vertex simplex algorithm.

We subject the data $\bar A$ and $\bar y$ to a random perturbation $\sigma A$, and $\sigma y$, where $A$ and $y$ have iid normal entries with mean zero and standard deviation $ \max_i\|(\bar y_i,\bar a_i)\|$. It is then shown in~\cite{Spielman2004} that the expected number of simplex steps is controlled by
\begin{gather*}
\mathbb E\big[T\big(\bar A + \sigma A, \bar y + \sigma y,z\big)\big] \leq P(1/\sigma,N,d),
\end{gather*}
where $P(a,b,c)$ is a polynomial. Thus, problems of polynomial complexity occupy a region of high probability.

\subsubsection{Smoothed analysis: LU factorization without pivoting}
Let $\bar A$ be an $N \times N$ non-singular matrix and consider computing its LU factorization, $\bar A = L U$,
without partial pivoting. The \emph{growth factor} of $\bar A$, def\/ined by
\begin{gather*}
\rho(\bar A) = \frac{\|U\|_\infty}{\|\bar A\|_\infty},
\end{gather*}
may be exponentially large in the size of the matrix, as seen in the following classical example:
\begin{gather*}
\bar A = \left[\begin{matrix} 1 & 0 & 0 & 0 &1\\
-1 & 1 & 0 & 0 & 1\\
-1 & -1 & 1 & 0 & 1\\
-1 & -1 & -1 & 1 & 1\\
-1 & -1 & -1 & -1 & 1 \end{matrix}\right] =\left[ \begin{matrix} 1 & 0 & 0 & 0 &0\\
-1 & 1 & 0 & 0 & 0\\
-1 & -1 & 1 & 0 & 0\\
-1 & -1 & -1 & 1 & 0\\
-1 & -1 & -1 & -1 & 1 \end{matrix}\right] \left[\begin{matrix} 1 & 0 & 0 & 0 &1\\
0 & 1 & 0 & 0 & 2\\
0 & 0 & 1 & 0 & 4\\
0 & 0 & 0 & 1 & 8\\
0 & 0 & 0 & 0 & 16 \end{matrix}\right].
\end{gather*}
Generalizing this example to all $N$, we see that $\rho(\bar A) = 2^{N-1}/N$. This is close to the worst-case estimate of Wilkinson~\cite{Wilkinson1961}. Now consider instead $\rho(\bar A + \sigma A)$ where the random perturbation~$A$ is an $N \times N$ matrix consisting of iid standard normal random variables. One of the results of~\cite{Sankar2006} is
\begin{gather}	\label{eq:smoothed-growth}
\mathbb P( \rho(\bar A + \sigma A) > 1 + t) \leq \frac{1}{\sqrt{2\pi}} \frac{N(N+1)}{\sigma t}.
\end{gather}
Hence the probability that $\rho(\bar A + \sigma A) \geq 2^{N-1}/N$ is exponentially small! The above estimate relies on a tail bound on the condition number
\begin{gather*}
\mathbb P(\kappa(\bar A + \sigma A) > t) \leq \frac{14.1 n \big(1 + \sqrt{2 (\log t)/9 n}\big)}{t \sigma}.
\end{gather*}

The example above may also be used to demonstrate exponential growth {\em with partial pivoting\/}. However, to the best of our knowledge, there are no smoothed analysis bounds analogous to~(\ref{eq:smoothed-growth}) that include the ef\/fect of pivoting.

\subsection{The main result}
We now formulate a notion of smoothed analysis for the halting time of the conjugate gradient algorithm. In order to do so, we must choose a matrix ensemble over which to take averages. Since the conjugate gradient algorithm is restricted to positive def\/inite matrices it is natural to choose random perturbations that are also positive def\/inite. The fundamental probability measure on Hermitian positive def\/inite matrices is the Laguerre unitary ensemble (LUE), or Wishart ensemble, def\/ined as follows. Assume $N$ is a positive integer and $\alpha > - N$ is another integer. Let $X$ be an $N \times (N + \alpha)$ matrix of iid standard complex normal random variables\footnote{A standard complex normal random variable is given by $Z = X + \I Y$ where $X$ and $Y$ are independent real normal random variables with mean zero and variance $1/2$.}. The Hermitian matrix $W=XX^*$ is an LUE matrix with parameter $1+\alpha/N$.

The parameter $\alpha$ plays an important role in our work. The case $\alpha=0$ is critical in the following sense. When $-N < \alpha < 0$, the random matrix $W = XX^*$ is positive {\em semi-definite} and~$0$ is an eigenvalue of multiplicity $-\alpha$ with probability 1. In particular, the condition number of~$W$ is inf\/inite almost surely. On the other hand, when $\alpha\geq 0$, the random matrix~$W$ is almost surely strictly positive def\/inite. When $\alpha = 0$, Edelman~\cite{Edelman1988} showed that the condition number of~$W$ is heavy-tailed, and does not have a f\/inite mean (see also~\cite{Sankar2006} and the previous examples). On the other hand, if $\alpha$ grows linearly with~$N$, say $\lim\limits_{N \to \infty}\alpha/N =p >0$, the leading-order asymptotics of the smallest eigenvalue of~$W$, and thus the condition number, are described by the Marcenko--Pastur distribution with parameter~$p$. In particular, as $N \to \infty$ the smallest eigenvalue of $W(N)/(N+\alpha)$ remains strictly separated from~$0$. In recent work with P.~Deift, we explored an intermediate regime $\alpha \sim 4cN^{1/2}$, and established Tracy--Widom \cite{TracyWidom} f\/luctuations of the smallest eigenvalue and the condition number (see~\cite[Theorems~1.1 and~1.3]{Deift2015}). We further showed numerically that the nontrivial f\/luctuations of the condition number are ref\/lected in the performance of the conjugate gradient algorithm on Wishart matrices in this regime. In this article, we broaden our exploration of this intermediate regime, choosing
\begin{gather}\label{e:scaling}
\alpha = \alpha(N) = \big\lfloor \sqrt{4c} N^\gamma \big\rfloor \qquad \text{for some} \quad 0 < \gamma \leq 1/2.
\end{gather}

In order to formulate a notion of smoothed analysis for the conjugate gradient algorithm, we must subject a deterministic positive-def\/inite matrix $A$ with $\|A\| \leq 1$ to a random perturbation of the form $\sigma^2 H$, where $\|H\| = O(1)$, and then take the supremum over all~$A$ with $\|A\| \leq 1$. It turns out that the largest eigenvalue of $W = XX^*$ is approximately $\nu = 4N + 2\alpha +2$. Thus, our implementation of smoothed analysis for the conjugate gradient algorithm involves estimating
\begin{gather*}
\sup_{A \geq 0, \, \|A\| \leq 1} \mathbb E[\tau_\epsilon(A + \sigma^2 H,b)], \qquad \sup_{A \geq 0, \, \|A\| \leq 1} \mathbb E[\tau_{w,\epsilon}(A + \sigma^2 H,b)], \qquad H = W/\nu,
\end{gather*}
with explicit dependence on $\sigma$ and $\gamma$. The factor $\sigma^2$ is used here so that $\sigma$ represents the scaling of the variance of the entries of~$X$. Our main result, proved in Section~\ref{s:tail}, is the following.

\begin{Theorem}\label{t:smooth}
 Assume $\alpha$ satisfies \eqref{e:scaling} and $\epsilon > 0$. Let $H = \nu^{-1}XX^*$ where $\nu = 4 N + 2 \alpha +2$ and $X$ is an $N \times (N+\alpha)$ matrix of iid standard complex normal random variables. Then with
 \begin{gather*}
 \rho_\sigma = 2\sqrt{\left(1 + \frac{1}{\sigma^2}\right) \frac{1}{c}}, \qquad \sigma > 0,
 \end{gather*}
 we have the following estimates.
 \begin{enumerate}[$(1)$]\itemsep=0pt
\item \textbf{Halting time with the $\ell^2$ norm:}
 \begin{gather*}
 \sup_{\|A\| \leq 1,\, A \geq 0} \mathbb E\big[\tau_{\epsilon}(A + \sigma^2 H,b)^j\big] \\
 \qquad{}\leq \frac{1}{2^j} N^{j(1-\gamma)}\rho_\sigma^j \big(\log N^{1-\gamma} \rho_\sigma 2\epsilon^{-1}\big)^j(1 + o(1)), \qquad \text{as} \quad N \to \infty.
 \end{gather*}
\item \textbf{Halting time with the weighted norm:}
 \begin{gather*}
 \sup_{\|A\| \leq 1,\, A \geq 0}\mathbb E\big[\tau_{w,\epsilon}(A + \sigma^2 H,b)^j\big] \leq \frac{1}{2^j} N^{j(1-\gamma)} \rho_\sigma^j\big(\log2\epsilon^{-1}\big)^j(1 + o(1)), \qquad \text{as} \quad N \to \infty.
 \end{gather*}
 \item \textbf{Successive residuals:} For $r_k = r_k(A + \sigma^2 H,b)$, the $k$th residual in the solution of $(A + \sigma^2 H)x = b$ with the conjugate gradient algorithm
 \begin{gather*}
 \sup_{\|A\| \leq 1,\,A \geq 0}\mathbb E\left[ \frac{\|r_{k+1}\|^j}{\|r_k\|^j}\right] \leq \left( 1 - \frac{2}{\rho_\sigma N^{1-\gamma}+1}(1+o(1))\right)^j, \qquad \text{as}\quad N \to \infty,
 \end{gather*}
 where $\|\cdot\|$ is either $\|\cdot\|_{\ell^2}$ or $\|\cdot\|_{w^{-1}}$.
 \end{enumerate}
 \end{Theorem}

\begin{Remark}
The parameters $(\gamma,\sigma)$ control the ef\/fect of the random perturbation in very dif\/ferent ways. In Lemma~\ref{l:tail} we precisely describe how increasing $\gamma$ leads to better conditioned problems. For all $0 < \gamma \leq 1/2$ (and conjecturally for $\gamma > 1/2$) the asymptotic size of the expectation and the standard deviation of $\tau_\epsilon$ is $\bigo(N^{1-\gamma} \log N)$, meaning that the conjugate gradient algorithm will terminate before its maximum of $N$ iterations with high probability. For instance, assume $\alpha$ satisf\/ies \eqref{e:scaling} with $c = 2$. We use Markov's inequality and Theorem~\ref{t:smooth} for $j > 0$ and suf\/f\/iciently large\footnote{Here $N$ should be suf\/f\/iciently large so as to make the error term in Theorem~\ref{t:smooth} less than unity.} $N$ to obtain
\begin{gather*}
\sup_{\|A\| \leq 1,\, A \geq 0} \mathbb P \big( \tau_\epsilon\big(A + \sigma^2 H,b\big) \geq N^{1-\lambda} \big) \\
\qquad{} \leq \frac{1}{2^{j-2}}\left(1 + \frac{1}{\sigma^2} \right)^{j/2}N^{-j(\gamma-\lambda)} \big(\log \big[4N^{1-\gamma}\big(1 + \sigma^{-2}\big)^{1/2} \epsilon^{-1}\big]\big)^j.
\end{gather*}
Hence for $\lambda < \gamma$ and $j$ large, this probability decays rapidly.
\end{Remark}

\begin{Remark}
We only prove Theorem~\ref{t:smooth} for LUE perturbations in the range $0 < \gamma \leq 1/2$. However, we expect Theorem~\ref{t:smooth} to hold for all $0 \leq \gamma \leq 1$, as illustrated in the numerical experiments below. In order to establish Theorem~\ref{t:smooth} in the range~$\frac{1}{2} < \gamma \leq 1$ it is only necessary to establish Lemma~\ref{l:tail} for these values of~$\gamma$. This will be the focus of future work.
\end{Remark}

\begin{Remark}
Theorem~\ref{t:smooth} provides aymptotic control on the $j$th moments of halting times for each~$j$. This formally suggests that
one may obtain a~bound on an exponential generating function of the halting times above. However, we cannot establish this because the condition number $\kappa$ has only $O(\alpha)$ moments at any f\/inite~$N$.
\end{Remark}

\begin{Remark}
By restricting attention to positive def\/inite perturbations we ensure that the conjugate gradient scheme is always well-def\/ined for the perturbed matrix $A + \sigma^2 H$. This also allows the following simple, but crucial lower bound, on the lowest eigenvalue of the perturbed matrix
\begin{gather*}
\lambda_{\min}\big(A + \sigma^2 H \big) \geq \lambda_{\min}(A),
\end{gather*}
which then yields an upper bound on the condition number of the perturbed matrix $\kappa (A +\sigma^2 H)$. We have not considered the question of random perturbations of $A$ that are Hermitian, but not necessarily positive def\/inite. Such perturbations are more subtle since they must be scaled according to the smallest eigenvalue of~$A$. Nor have we considered the question of whether such perturbations provide good `real-life' models of a smoothed analysis of the conjugate gradient scheme. Nevertheless, the above framework shares important features with~\cite{Sankar2006} in that the problem is ``easier'' for large values of $\sigma$ and the worst case of the supremum over the set $\{A \geq 0, \; \|A\| \leq 1\}$ can be realized at singular~$A$.
\end{Remark}

\subsection{Numerical simulations and the accuracy of the estimates}\label{s:numerics}

In this section we investigate how close our estimates on $\mathbb E[\tau_{\epsilon}(A + \sigma^2 H,b)]$ are to the true value of the expectation. We present numerical evidence that in the ``$\sigma = \infty$'' (also obtained by $A = 0$ and $\sigma = 1$) case the estimates are better for larger values of $\gamma$, and continue to hold beyond the $\gamma = 1/2$ threshold of Theorem~\ref{t:smooth}. We also give examples for specif\/ic choices of~$A$ and demonstrate that, as expected, the actual behavior of the conjugate gradient algorithm is much more complicated for $A \neq 0$.

Because the conjugate gradient algorithm is notoriously af\/fected by round-of\/f error, we adopt the following approach to simulating $\tau_{\epsilon}(M,b)$, $M = A + \sigma^2 H$ with f\/inite-precision arithmetic:
\begin{itemize}\itemsep=0pt
\item In exact arithmetic, the conjugate gradient algorithm applied to $Mx = b$, $M = U^* \Lambda U$, with initial guess $x_0 = 0$, has the same residuals as the algorithm applied to $\Lambda y = U^* b$. Indeed, if $x_k$ satisf\/ies $\|Mx_k - b\|_{w^{-1}} = \min\limits_{x \in \mathcal K_k} \|Mx - b\|_{w^{-1}}$ then for $y = U^* x$, $y_k = U^* x_k$, $\|Mx_k - b\|_{w^{-1}}^2 = \langle Mx_k - b, x_k - x\rangle = \langle \Lambda y_k - U^*b, y_k - y \rangle $. Thus, def\/ining $\|\cdot\|^2_{\tilde w^{-1}} = \langle \cdot,\Lambda^{-1} \cdot\rangle$ we have
	 \begin{gather*}
	 \|\Lambda y_k - U^* b\|_{\tilde w^{-1}} = \min_{y \in \tilde {\mathcal K}_k} \|\Lambda y - U^* b\|_{\tilde w^{-1}},\\
\tilde {\mathcal K}_k = U^* \mathcal K_k = \operatorname{span}\big\{U^*b, \Lambda U^*b, \ldots, \Lambda^{k-1} U^*b \big\}.
	 \end{gather*}
	This is an exact characterization of the iterates of the conjugate gradient algorithm applied to $\Lambda y = U^*b$.
\item Sample a matrix $H = XX^*/\nu$ and compute the spectral decomposition $A + \sigma^2 H = U\Lambda U^*$. Sample a vector $b$ with iid Gaussian entries and normalize\footnote{Choosing $b$ in this way is convenient for these manipulations but it is not necessary. We choose a non-Gaussian vector for our actual experiments.} it, so that $\|b\|_{\ell^2} = 1$. 	Prior to normalization the entries of $b$ are iid Gaussian, thus $b$ is uniformly distributed on the unit sphere in $(\mathbb{C}^N,\ell^2)$. Note that if $A = 0$, $M = \sigma^2 H$ is a Wishart matrix, we f\/ind that $U^*b$ is also uniformly distributed on the unit sphere in $(\mathbb{C}^N,\ell^2)$. That is, $b$ and $U^*b$ have the same law.
\item Applying the diagonal matrix $\Lambda$ to a vector is much less prone to round-of\/f error since it involves only $N$ multiplications, as opposed to $N^2$ multiplications for the dense matrix~$H$. Thus, to minimize round-of\/f error we compute the iterates of the conjugate gradient algorithm applied to $\Lambda y=b$ with $\Lambda$ as above, and $b$ uniformly distributed on the unit sphere in $(\mathbb{C}^N,\ell^2)$. As noted above when $A = 0$, these iterates have the same law as those of $Hx =\tilde{b}$ when $\tilde{b}$ and $b$ have the same law, and when $H$ is a Wishart matrix. By computing the number of iterations necessary (in high-precision arithmetic) so that $\|y_{k+1}\|_{\ell^2} \leq \epsilon$, we obtain one sample of the halting time $\tau_{\epsilon}(M,b)$ without signif\/icant round-of\/f errors.
\end{itemize}

\subsubsection[The ``$\sigma = \infty$'' case]{The ``$\boldsymbol{\sigma = \infty}$'' case}

Now, we investigate how close our estimates on $\mathbb E[\tau_{\epsilon}(H,b)]$ (which can be obtained from Theorem~\ref{t:smooth} by formally sending $\sigma \to \infty$). In Figs.~\ref{f:cLUE}, \ref{f:nLUE} and \ref{f:sLUE} we plot the sample mean $\overline{\tau_{\epsilon}}$ over 1,000 samples as $N$ increases. Throughout our numerical experiments $b$ is taken to be iid uniform on $[-1,1]$ and then normalized to be a unit vector. With this consideration, it is clear that the estimate in Theorem~\ref{t:smooth} is good for $\gamma = 2/3$ (despite the fact that we have not proved it holds in this case), fairly tight for $\gamma = 1/2$ and not as good for $\gamma = 1/3$. These calculations demonstrate that the worst case bounds~\eqref{rutA} and~\eqref{rut} provide surprisingly good estimates in a~random setting. Further, they appear to be exact in the sense that Theorem~\ref{t:smooth} predicts the correct order of the expectation of~$\tau_{\epsilon}$ as~$N \to \infty$.

\begin{figure}[th!]\centering
\subfigure[]{\includegraphics[width=.45\linewidth]{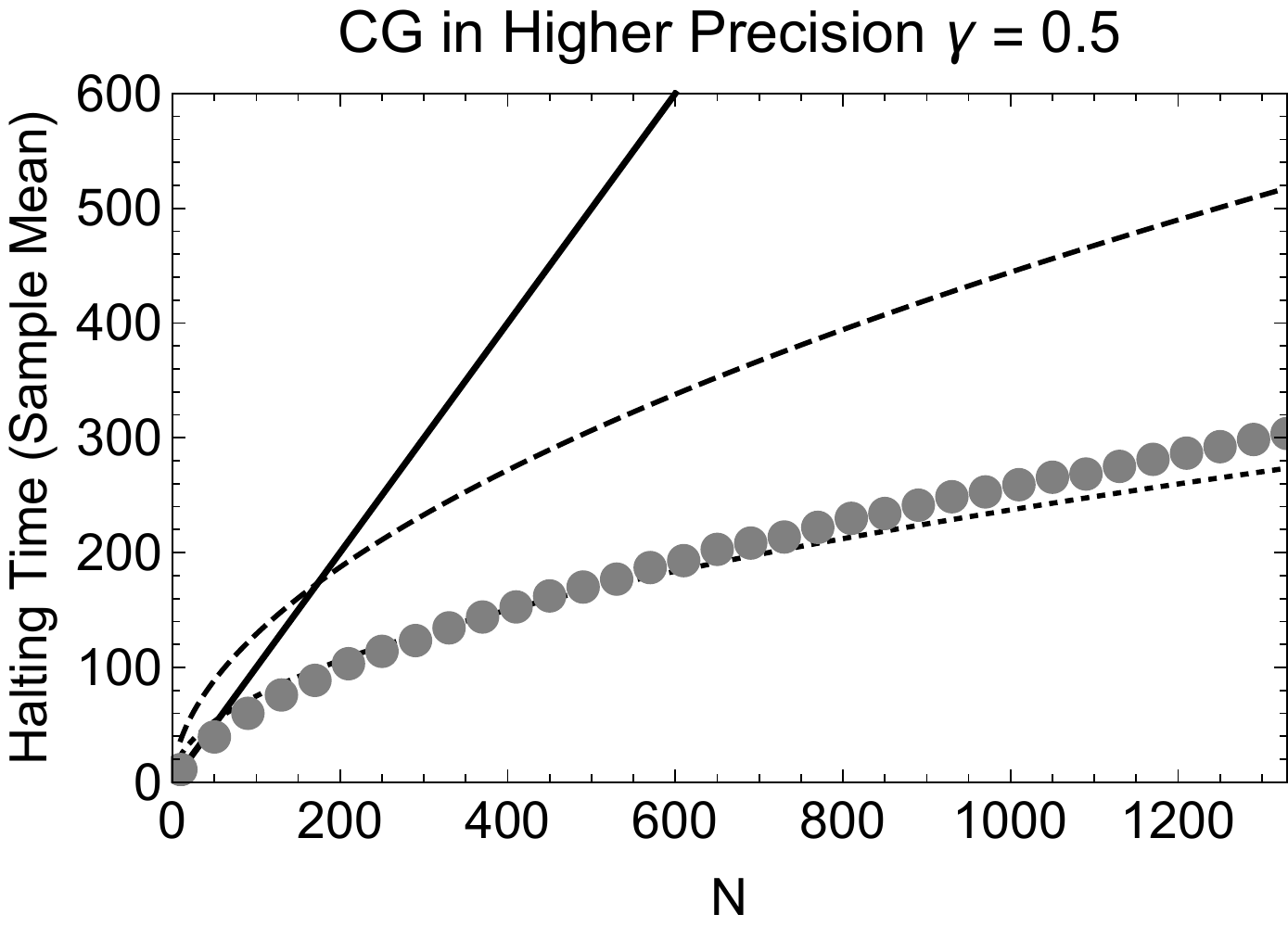}\label{f:cplt}}\hfill
\subfigure[]{\includegraphics[width=.45\linewidth]{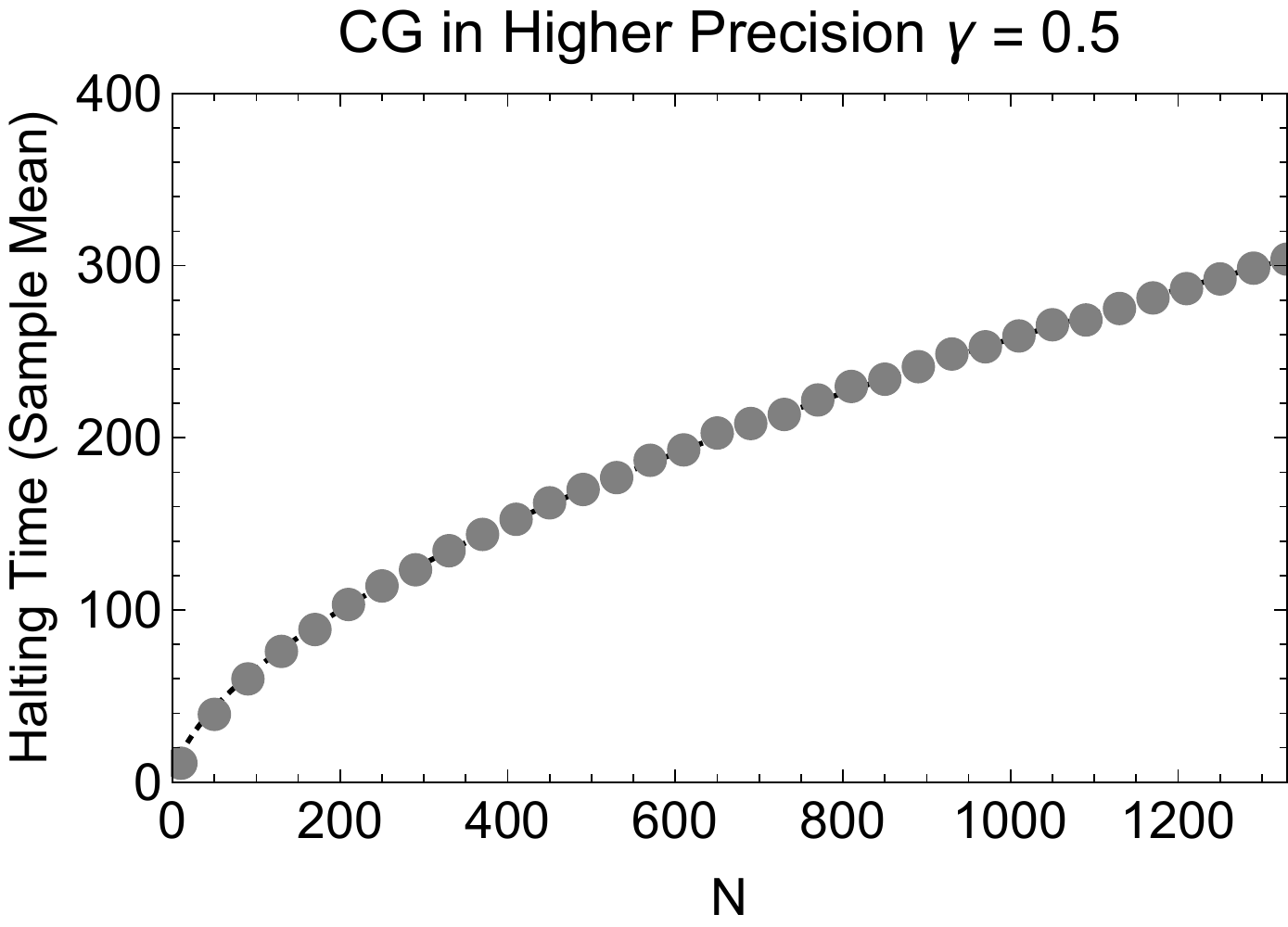}\label{f:cplt2}}
\caption{(a)~The sample mean $\overline{\tau_{\epsilon}}$ as a function of $N$ for $\gamma = 1/2$ and $c = 1$ with $\epsilon = 10^{-4}$. The plot also shows the deterministic maximum of $N$ iterations for the conjugate gradient algorithm (solid line), the upper bound computed in Theorem~\ref{t:smooth} (dashed line) and the curve $7.5 N^{1/2}$ (dotted line) to demonstrate that $\overline{\tau_{\epsilon}}$ grows faster than $N^{1/2}$. (b)~A~f\/it of the data points using the function $F(N) = a N^{1/2} \log N + b N^{1/2}$, $a \geq 0$ and $b \geq 0$ plotted against the data. Here we f\/ind $a = 0.67$ and $b = 3.51$ indicating that Theorem~\ref{t:smooth} predicts the correct scaling in $N$ for $\overline{\tau_{\epsilon}}$.}\label{f:cLUE}
\end{figure}

\begin{figure}[th!]\centering
\subfigure[]{\includegraphics[width=.45\linewidth]{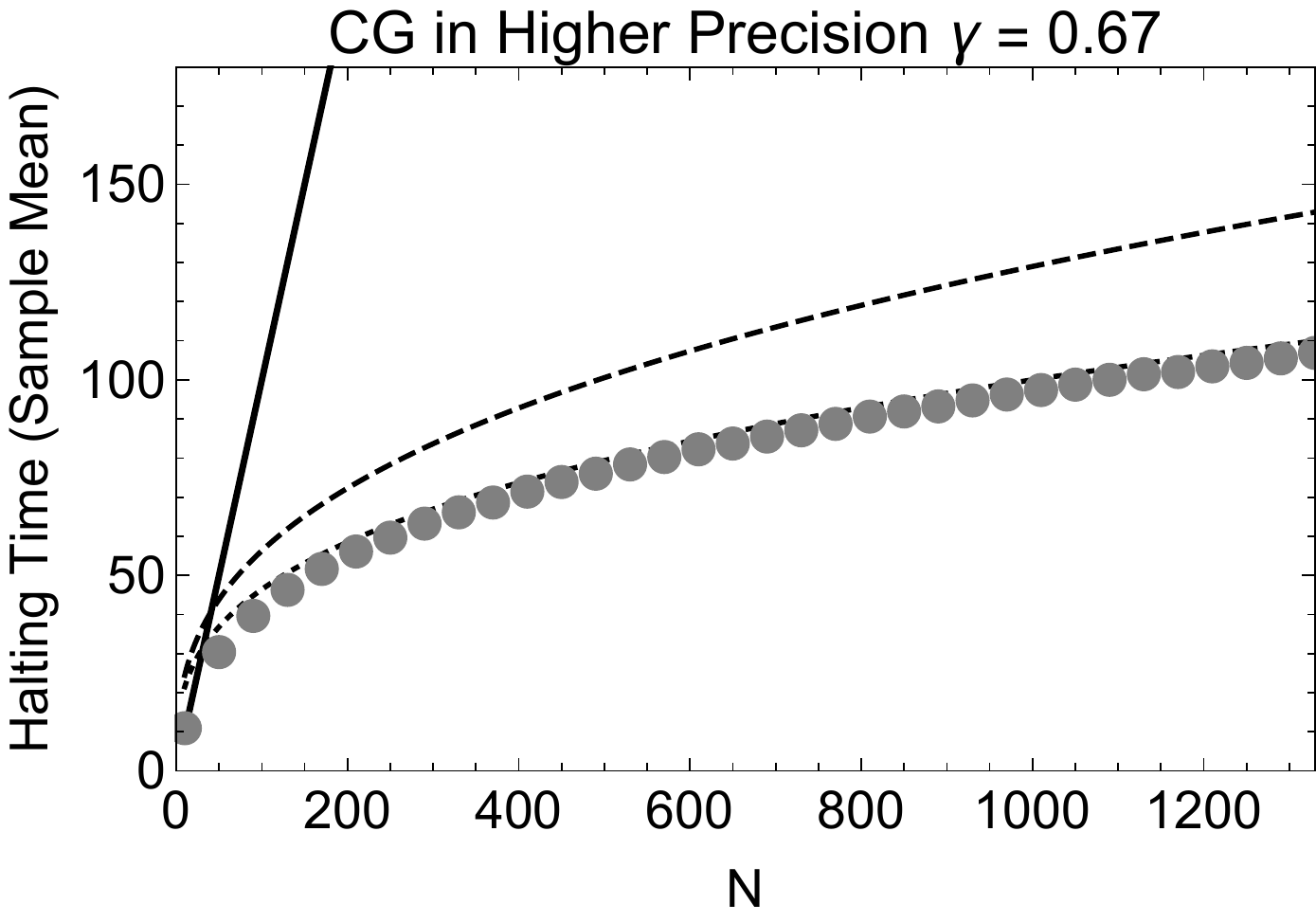}\label{f:nplt}}\hfill
\subfigure[]{\includegraphics[width=.45\linewidth]{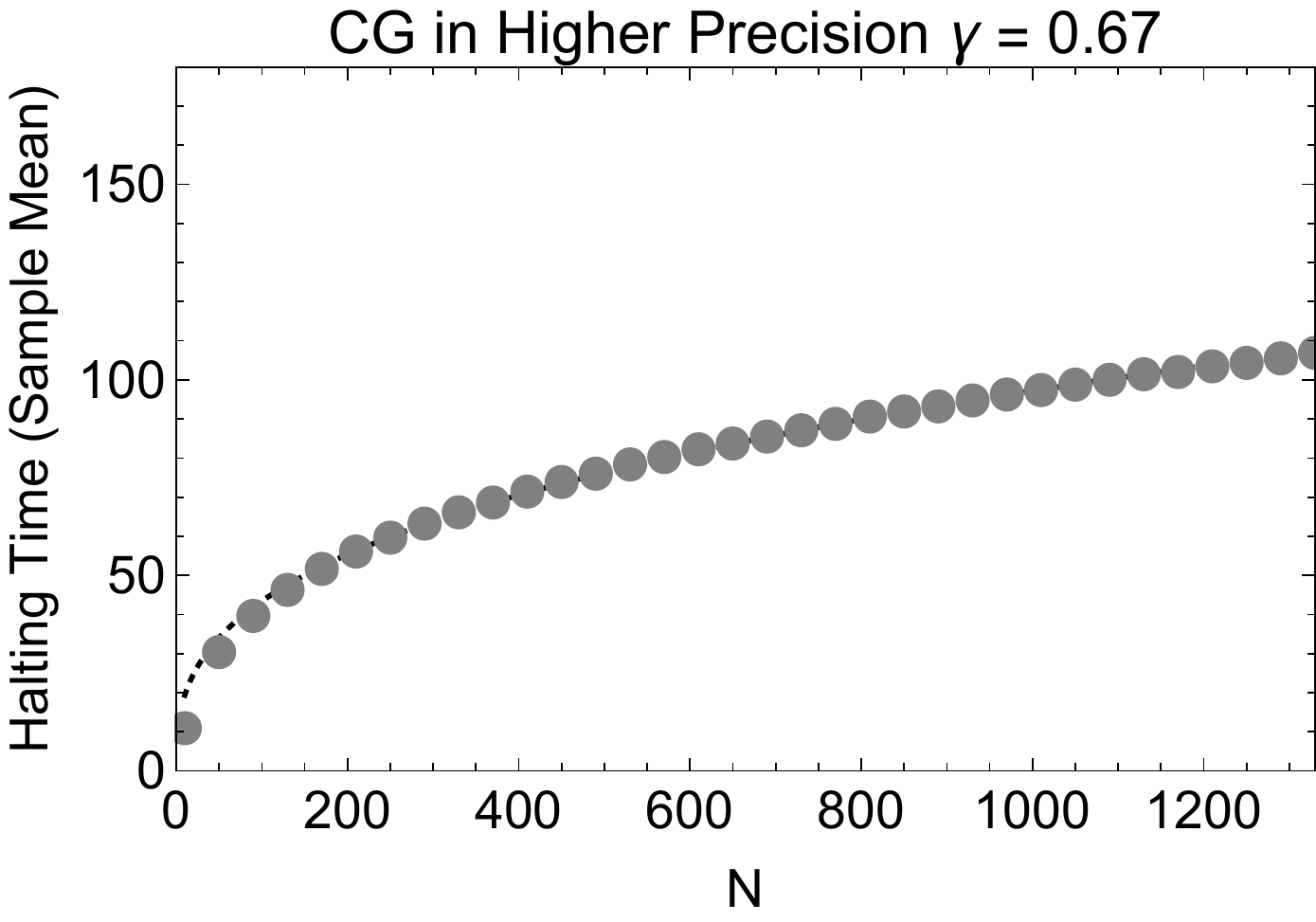}\label{f:nplt2}}
\caption{(a)~The sample mean $\overline{\tau_{\epsilon}}$ as a function of $N$ for $\gamma = 2/3$ and $c = 1$ with $\epsilon = 10^{-4}$. The plot includes the linear upper bound $\tau_\epsilon \leq N$ (solid line), the upper bound computed in Theorem~\ref{t:smooth} (dashed line) and the curve $10 N^{1/3}$ (dotted line) to demonstrate that $\overline{\tau_{\epsilon}}$ grows at approximately $N^{1/3}$. (b)~A~f\/it of the data points using the function $F(N) = a N^{1/3} \log N + b N^{1/3}$, $a \geq 0$ and $b \geq 0$ plotted against the data. Here we f\/ind $a = 0.152$ and $b = 8.66$. Again, Theorem~\ref{t:smooth} predicts the correct scaling for $\overline{\tau_{\epsilon}}$.}\label{f:nLUE}
\end{figure}

\begin{figure}[th!]\centering
\subfigure[]{\includegraphics[width=.45\linewidth]{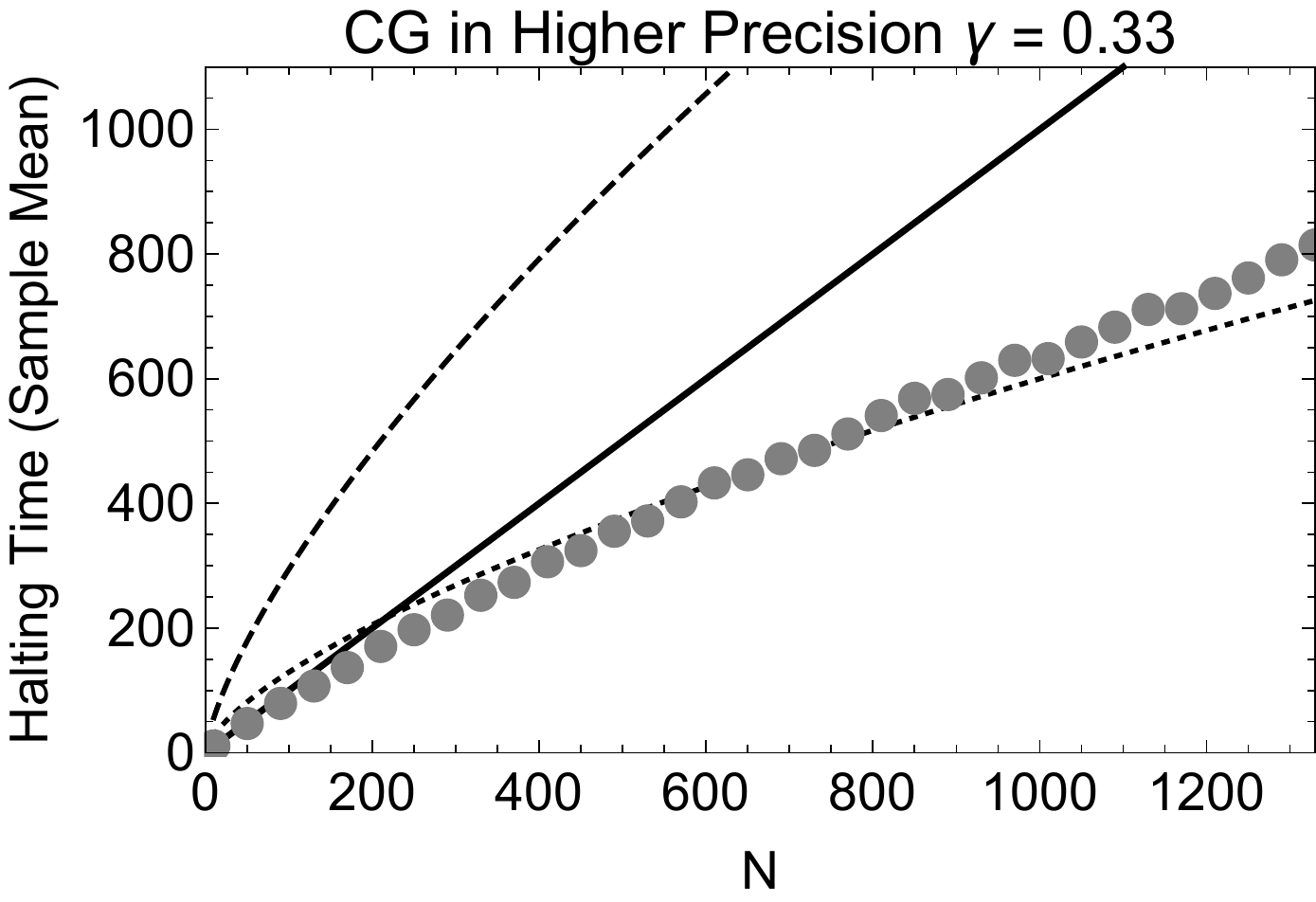}\label{f:splt}}\hfill
\subfigure[]{\includegraphics[width=.45\linewidth]{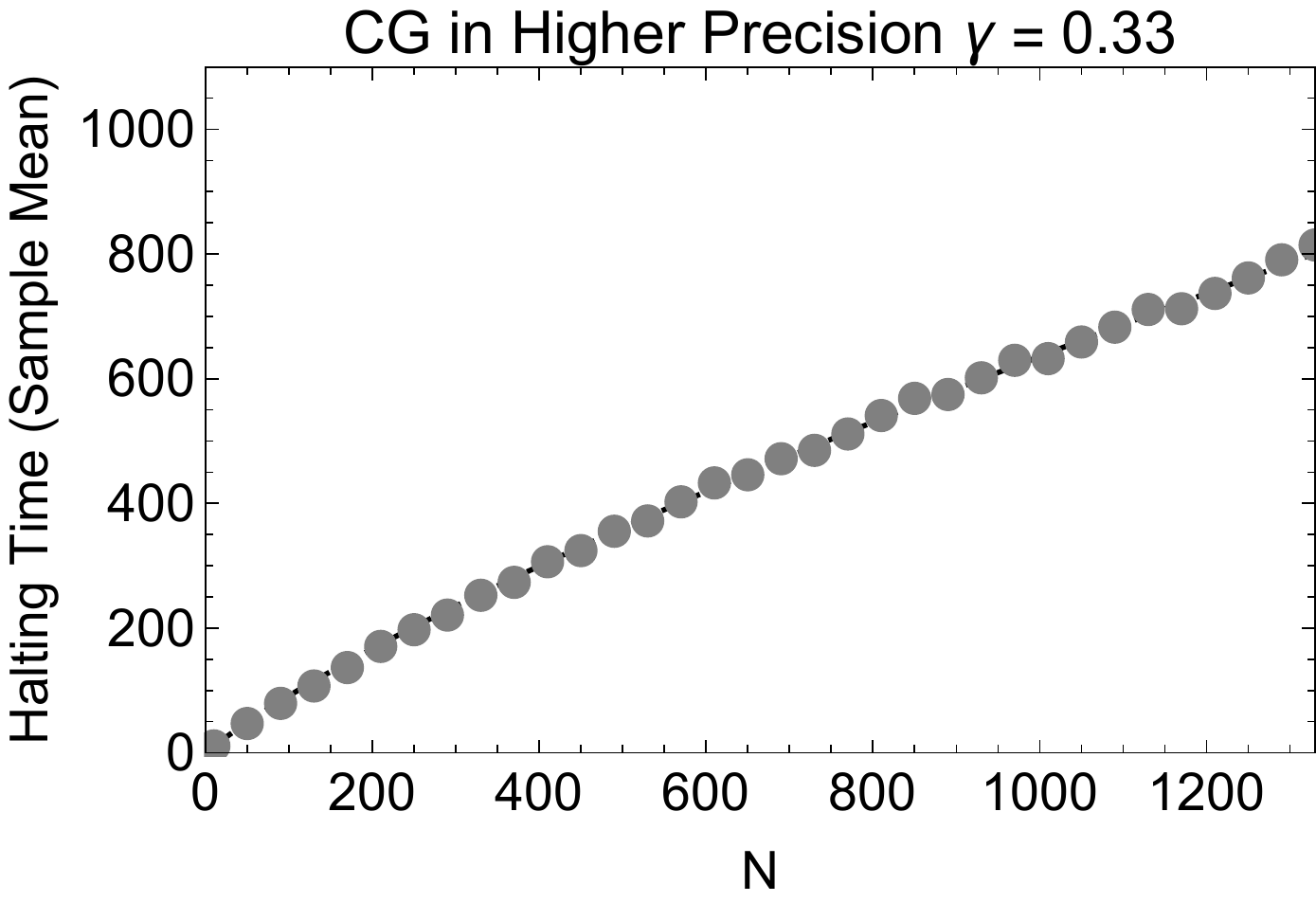}\label{f:splt2}}
\caption{(a)~The sample mean $\overline{\tau_{\epsilon}}$ as a function of $N$ for $\gamma = 1/3$ and $c = 1$ with $\epsilon = 10^{-4}$. The plot also shows the upper bound of $N$ iterations (solid line), the upper bound computed in Theorem~\ref{t:smooth} (dashed line) and the curve $6 N^{2/3}$ (dashed) line to demonstrate that $\overline{\tau_{\epsilon}}$ grows faster than $N^{2/3}$. (b)~A~f\/it of the data points using the function $F(N) = a N^{2/3} \log N + b N^{2/3}$, $a \geq 0$ and $b \geq 0$ plotted against the data. The parameters $a = 0.916$ and $b = 0$ f\/it the data very well.}\label{f:sLUE}
\end{figure}

Comparing these numerical results with Theorem~\ref{t:smooth} we conclude that:
 \begin{enumerate}[(1)]\itemsep=0pt
 \item Tail estimates on the condition number derived from tail estimates of the extreme eigenvalues, can be used to obtain near optimal, and in some cases optimal, estimates for the expected moments of the condition number.
 \item In light of rigorous results and heuristic expectations of universality in random matrix theory, we f\/ind it reasonable to expect Lemma~\ref{l:tail} and Theorem~\ref{t:smooth} to hold for more general real and complex sample covariance matrices, not just LUE.
 \item The worst-case estimates given in~\eqref{rut} and~\eqref{rutA} produce ef\/fective bounds on the moments of the halting time, and predict the correct order of growth of the mean as $N \to \infty$. The importance of this observation is that these bounds are known to be sub-optimal. Thus, our results show that the matrices for which these estimates are sub-optimal have a small probability of occurrence.
 \end{enumerate}

\subsubsection{Perturbed discrete Laplacian}
The numerical examples of the previous section are dominated by noise. In this subsection and the next, we investigate the ef\/fect of small LUE perturbations on structured matrices~$A$. This is a more subtle problem since it is hard to conjecture the growth rate of $\mathbb E[\tau_{\epsilon}(A + \sigma^2 H,b)]$ as~$\sigma$ and~$\gamma$ vary for a~given~$A$. We present numerical experiments on random perturbations of two examples that have been studied in the literature on the conjugate gradient algorithm~-- discrete Laplacians and singular matrices with clusters of eigenvalues. In both these examples, we numerically estimate the growth with $N$ of the halting time $\tau_{\epsilon}(A + \sigma^2 H,b)$ for a range of~$\sigma$ and $\gamma$. These numerical computations are compared with the unperturbed ($\sigma = 0$) and noise dominated (``$\sigma = \infty$'') cases. Broadly, we observe that f\/inite noise gives faster convergence (smaller halting time) with a dif\/ferent scaling than what is expected with no noise. We also f\/ind that when~$A \neq 0$, the halting time is not strongly af\/fected by~$\gamma$. At present, these are numerical observations, not theorems. We hope to investigate the accelerated convergence provided by noise in future work.

\begin{figure}[t!]\centering
\subfigure[]{\includegraphics[width=.45\linewidth]{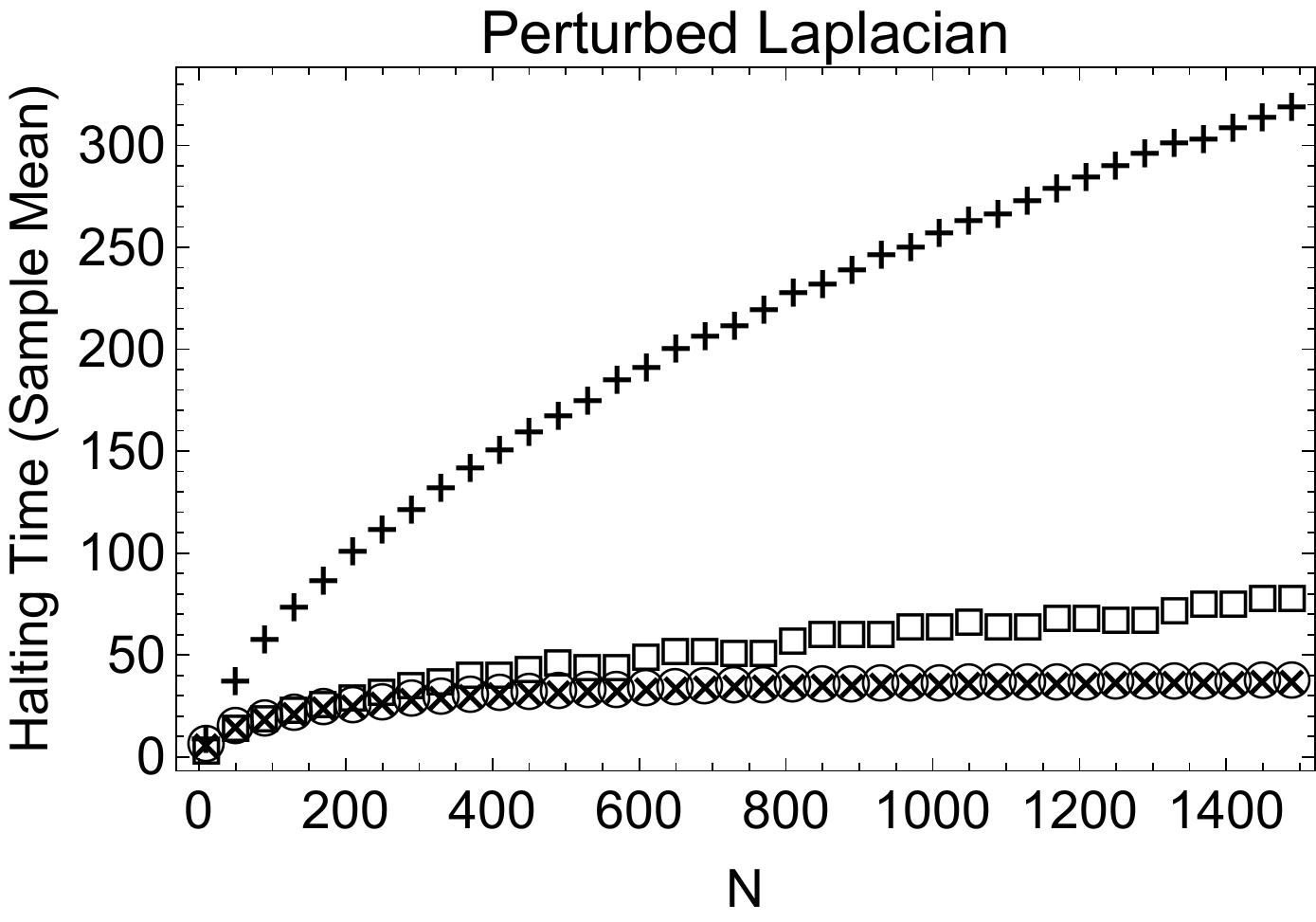}\label{f:dschr1}}\hfill
\subfigure[]{\includegraphics[width=.45\linewidth]{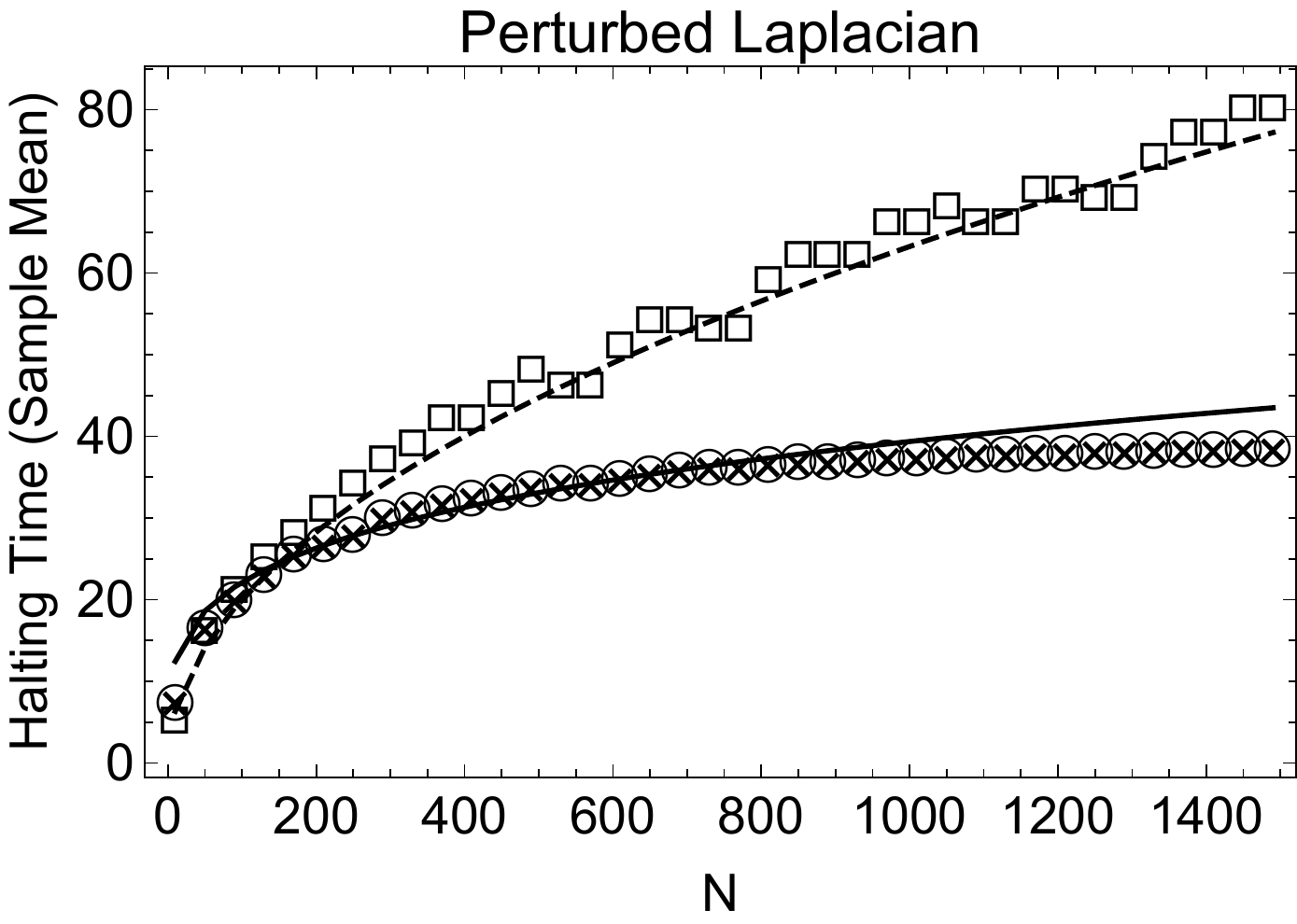}\label{f:dschr2}}
\caption{Numerical computations for sample mean of the halting time in the case of a~randomly perturbed discrete Laplacian operator $-\Delta_{m,k} + \sigma^2 H$. We let $m = k = \lfloor \sqrt{N} \rfloor$ and make the following choices of parameter values $\sigma = 0$~($\Box$), $\sigma = 0.1$, $\gamma = 1/2$~($\circ$), $\sigma = 0.1$, $\gamma = 1/3$~($\times$) and $\sigma = \infty$, $\gamma = 1/2$~($+$). (a)~The sample mean for all four parameter choices. (b)~Three parameter choices plotted along with the functions $2 \sqrt{N}$ (dashed) and $7 N^{1/4}$ (solid).}\label{f:dschr}
\end{figure}

In our f\/irst example, $A$ is the $mk \times mk$ 2D discrete Laplacian def\/ined by the Kronecker product
\begin{gather*}
\Delta_{m,k} = I_{m} \otimes D_{2,k} + D_{2,m} \otimes I_{k},
\end{gather*}
where $D_{2,m}$ is the $m\times m$ symmetric tridiagonal matrix with $-2$ on the diagonal and $1$ on the of\/f-diagonals. We choose $m = k = \lfloor \sqrt{N} \rfloor $ in the computations.

Some results of numerical experiments with this choice of $A$ are shown in Fig.~\ref{f:dschr}. The scaling of the sample mean of the halting time, $\overline{\tau_{\epsilon}}$, is $O(\sqrt{N})$ in the extreme cases when $\sigma=0$ or $\sigma=\infty$ (see $+$ and $\Box$ in Fig.~\ref{f:dschr}). However, when $\sigma$ is $O(1)$, we f\/ind that $\overline{\tau_{\epsilon}} \sim N^{1/4}$ (see $\circ$ and $\times$ in Fig.~\ref{f:dschr}). Further, this result is not sensitive to $\gamma$. Therefore there is a complicated relationship between the deterministic matrix $A$, the random perturbation $H$ and the halting time that is not captured by Theorem~\ref{t:smooth}.

\subsubsection{Perturbed eigenvalue ``clusters''}
In our second example, we consider random perturbations of a singular matrix with clusters of eigenvalues. This construction is motivated by \cite[Section~5.6.5]{Liesen2013} and~\cite{Greenbaum1992}.

We def\/ine $A$ to be the $mk \times mk$ diagonal matrix obtained by sampling the Marchenko--Pastur law as follows\footnote{We use the Marchenko--Pastur law because it gives the asymptotic density of eigenvalues of the LUE we are considering~\cite{Marcenko1967}.}. Let $\zeta_{j,k}$, $j = 1, \ldots, k$ be def\/ined by
\begin{gather*}
\zeta_{j} = \zeta_{j,k} = \min\left\{ 0 \leq t \leq 4\colon \int_{-\infty}^t \frac{1}{2\pi} \sqrt{\frac{4-x}{x}} dx = \frac{j}{k} \right\}.
\end{gather*}
Then we def\/ine for $1 \leq j \leq k, 1 \leq \ell \leq m$
\begin{gather}\label{e:spread}
 \lambda_{\ell,j} = \begin{cases} 0, & j = 0,\\
 \zeta_j + \left(\dfrac{\ell - \lfloor m/2 \rfloor}{m}\right) \dfrac{1}{10k^2}, & j > 0. \end{cases}
\end{gather}
Finally, set $A=M_{m,k} = \diag(\lambda_{\ell,j})$ with any (consistent) ordering.

This produces a $mk \times mk$ diagonal matrix with $m$ zero eigenvalues, and $k(m-1)$ eigenvalues that are each clustered at quantiles of the Marchenko--Pastur law. Note that
\begin{gather*}
\frac{1}{k} = \int_{-\infty}^{\zeta_{1,k}} \frac{1}{2\pi} \sqrt{\frac{4-x}{x}} dx \sim c \zeta_{1,k}^{1/2}, \qquad k \to \infty,
\end{gather*}
and so $\zeta_{1,k} = \bigo(k^{-2})$. We divide by $k^2$ in \eqref{e:spread} to ensure we maintain a positive semi-def\/inite matrix. As in the previous section we set $m = k = \lfloor \sqrt{N} \rfloor $ in the computations and plot similar sample mean results in Fig.~\ref{f:clusters}. Despite the fact that $M_{m,k}$ is a singular matrix, the conjugate gradient algorithm converges rapidly for the perturbed matrix $M_{m,k} + \sigma^2 H$. In particular, Fig.~\ref{f:clusters2} shows a rate of growth that is only $O(\sqrt{\log N})$.

Finally, in the construction of $M_{m,k}$ we imposed the condition that $m$ eigenvalues are zero. If we considered a case where more and more eigenvalues are set to zero as $m,k \to \infty$ we would expect a transition to the $\sigma = \infty$ case.

\begin{figure}[th!]\centering
\subfigure[]{\includegraphics[width=.45\linewidth]{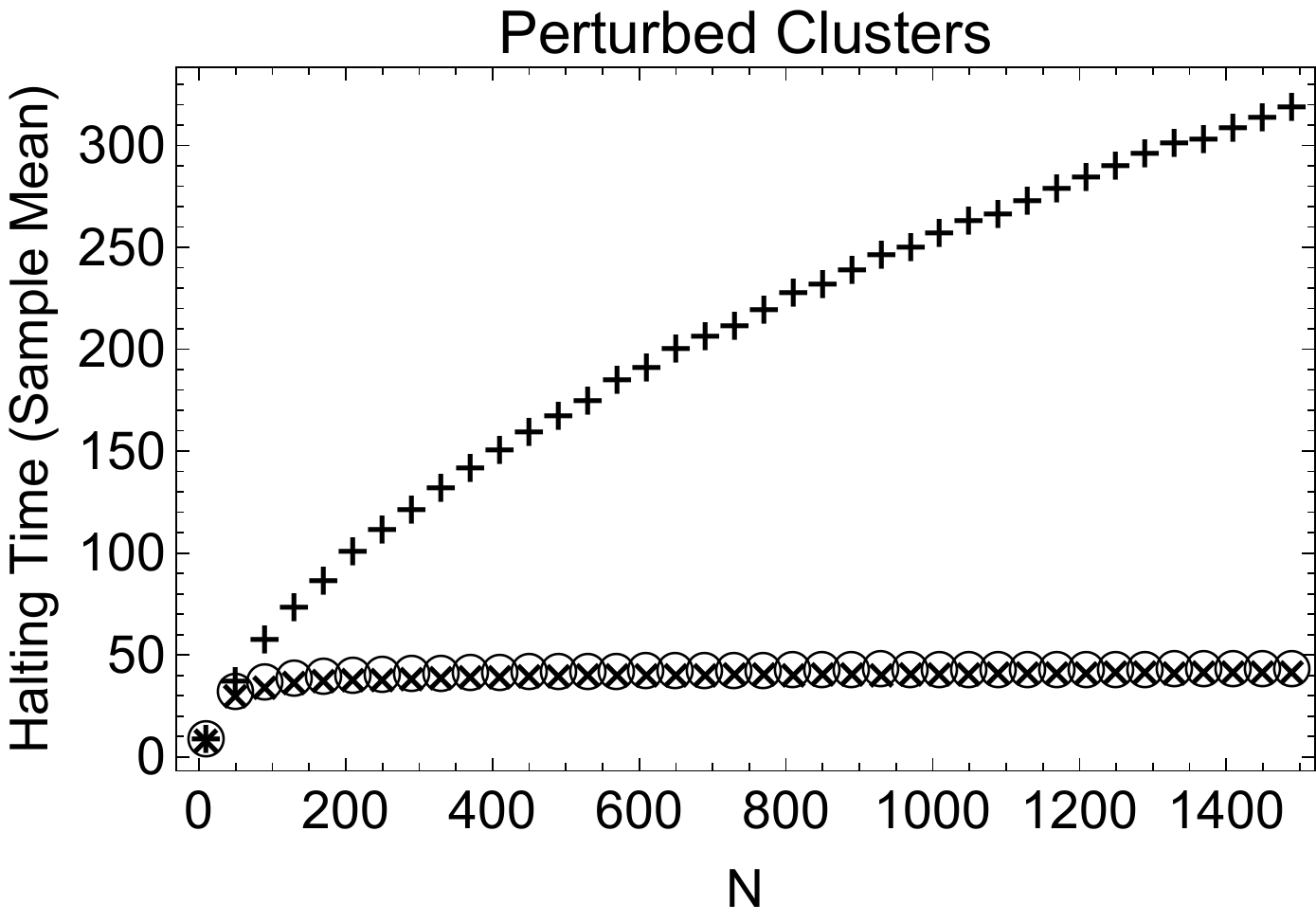}\label{f:clusters1}}\hfill
\subfigure[]{\includegraphics[width=.45\linewidth]{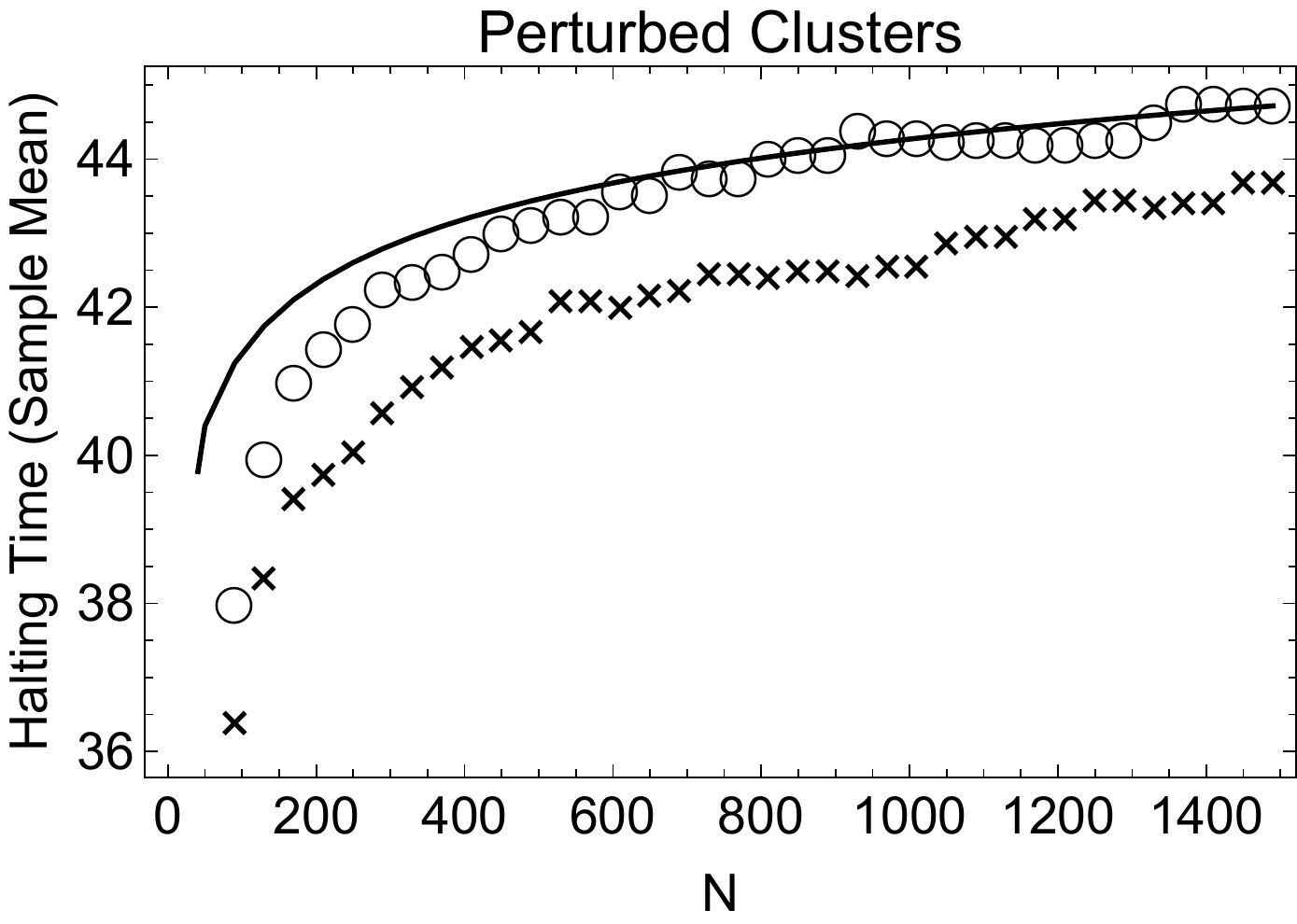}\label{f:clusters2}}
\caption{Numerical computations for sample mean of the halting time in the case of a randomly perturbed eigenvalue clusters: $M_{m,k} + \sigma^2 H$. Again, we let $m = k = \lfloor \sqrt{N} \rfloor$ and make the following choices of parameter values $\sigma = 0.1$, $\gamma = 1/2$~($\circ$), $\sigma = 0.1$, $\gamma = 1/3$ ($\times$) and $\sigma = \infty$, $\gamma = 1/2$~($+$). (a)~The sample mean for all four parameter choices. (b)~Three parameter choices plotted along with the function $6 \log^{1/2} (1+N) + 28.5$ to demonstrate how slow the halting time grows for $0 < \sigma < \infty$.}\label{f:clusters}
\end{figure}

\section{Estimating the halting time}

\subsection{Outline of the proof}

In this section, we explain the main steps in the proof of Theorem~\ref{t:smooth}. We also abstract the properties that are known to hold for LUE perturbations (as established in Section~\ref{s:LUE}), stating these estimates as a general condition on the tails of the smallest and largest eigenvalues that suf\/f\/ice to prove Theorem~\ref{t:smooth}.

In order to explain the main idea, we focus on controlling the halting time $\tau_{w,\epsilon}$ using estimate~\eqref{rut}. For brevity, let us def\/ine the parameter
\begin{gather*}
\theta(\kappa) = \frac{\sqrt{\kappa}-1}{\sqrt{\kappa}+1}.
\end{gather*}
Since $\kappa \geq 1$, the parameter $\theta(\kappa) \in [0,1)$. Let us also def\/ine the positive real number
\begin{gather*}
K_\epsilon (\kappa) = \inf_{k >0} \big\{\theta(\kappa)^k \leq \epsilon/2\big\} = \frac{\log \epsilon/2}{\log \theta(\kappa)}.
\end{gather*}
It follows immediately from \eqref{haltA} and the normalization $\|r_0\|_{\ell^2}=1$ that $\tau_{w,\epsilon} \leq K_\epsilon$, so that for every $a \geq 0$,
\begin{gather*}
\mathbb{P}\left(\tau_{w,\epsilon}>a \right)	 \leq \mathbb{P} \left(K_\epsilon >a\right) =
\mathbb{P} \left( \log \epsilon/2 < a \log \theta(\kappa)\right).
\end{gather*}
Note that $\log \theta(\kappa)<0$ and that as $\kappa \to \infty$,
\begin{gather*}
\log \theta(\kappa) \sim -\frac{2}{\sqrt{\kappa}}.
\end{gather*}
Thus, basic convergence properties of the conjugate gradient algorithm may be obtained from tail bounds on the condition number. Finally, the condition number is estimated as follows. Since $\kappa = \lambda_{\max}/\lambda_{\min}$ it is clear that upper bounds of the form $\mathbb{P}(\lambda_{\max}>t)$ and $\mathbb{P}(\lambda_{\min}^{-1}>t)$ for arbitrary $t \in (0,\infty)$ may be combined to yield an upper bound on $\mathbb{P}(\kappa>a)$ by suitably choosing~$t(a)$. As noted in the f\/irst three lines of the proof of Theorem~\ref{t:smooth} below, estimates of upper and lower eigenvalues for Wishart ensembles established in~\cite{Deift2015} immediately extend to estimates for matrices of the form~$A + \sigma^2H$.

\subsection{A general suf\/f\/icient condition}
The abstract property we use to establish Theorem~\ref{t:smooth} is the following.

\begin{Condition}\label{cond:tails}
Given a random positive-definite matrix~$H$, assume there exist positive constants~$c_1$ and $\delta$, constants $C_1$, $C_2$, and $a$ that are greater than $1$, and a positive function $f\colon (0,\infty) \to (0,\infty)$ such that
\begin{gather}
T_{\max}(t) := \mathbb P(\lambda_{\max}(H) > t) \leq \tilde T_{\max}(t) := C_1 \E^{-c_1N(t-a)}, \qquad t \geq 1,\notag\\
T_{\min}(t) := \mathbb P(\lambda_{\min}^{-1}(H) > t) \leq \tilde T_{\min}(t) := C_2 [t/f(N)]^{-\alpha/2}, \qquad t \geq (1+\delta)f(N).\label{e:min}
\end{gather}
Assume further that $T_{\max/\min}$ are strictly monotone functions of $t$ and $\lim\limits_{N \to \infty} f(N) = \infty$.
\end{Condition}
While the conditions above seem arbitrary at f\/irst sight, we will show how they emerge naturally for the LUE ensemble in the next section. In particular, we show that these conditions are satisf\/ied by a class of LUE matrices in Lemmas~\ref{l:LUEest} and~\ref{l:tail}.

\begin{Lemma}\label{l:cond-tail}
Assume Condition~{\rm \ref{cond:tails}} and that $\alpha$ grows with $N$ as in~\eqref{e:scaling}. Then there exists a~constant $C > 0$ such that
\begin{gather*}
\mathbb P(\kappa(H) > t) \leq C \left[a^{-1} t/f(N) \right]^{-\alpha/2 + e_N}, \qquad e_N = \half \frac{\alpha^2}{2 ac_1 N + \alpha},
\end{gather*}
when $t \geq a(1+\delta_N) f(N)$ where
\begin{gather*}
\delta_N = (1+\delta)\big(1 + C\big(N^{-1} + \delta N^{\gamma-1}\big)\big)-1.
\end{gather*}
\end{Lemma}

\begin{proof}
First, if $xy > ab$ and all these numbers are positive, then either $x > a$ or $y > b$. Thus for $0 \leq s \leq 1$, the tails bounds of Condition~\ref{cond:tails} imply
\begin{gather*}
\mathbb P\left( \frac{\lambda_{\max}(H)}{\lambda_{\min}(H)} > T^{-1}_{\max}(s)T^{-1}_{\min}(s)\right) \leq 2 s.
\end{gather*}
This bound may be `inverted' in the following way. If we def\/ine $T^{-1}(s) = T^{-1}_{\max}(s)T^{-1}_{\max}(s)$, then
\begin{gather*}
\mathbb P\left( \frac{\lambda_{\max}(H)}{\lambda_{\min}(H)} > t\right) \leq 2 T(t).
\end{gather*}

Our goal is to obtain an upper bound on $T(t)$ using the upper bounds $\tilde{T}_{\max}$ and $\tilde{T}_{\min}$. If $T_{\max}(t_1) = s$ and $\tilde T_{\max}(t_2) = s$ then $t_2 \geq t_1$. Therefore,
\begin{gather*}
T^{-1}_{\max}(s) \leq a + \frac{\log C_1/s}{c_1 N}, \qquad s \leq T_{\max}(a),\\
T^{-1}_{\min}(s) \leq {f(N)}\left[\frac{s}{C_2}\right]^{-2/\alpha}, \qquad s \leq T_{\min}((1+\delta)f(N)).
\end{gather*}
Since $(1 + \log t/n)^n \leq t$ when $t>1$ and $n \geq 1$, we can estimate
\begin{gather*}
a + \frac{\log C_1/s}{c_1 N} = a \left( 1+ \frac{\log C_1/s}{a c_1 N} \right)\leq a\left(\frac{C_1}{s}\right)^{1/(ac_1N)}.
\end{gather*}
Hence
\begin{gather*}
T^{-1}(s) \leq aC^{1/(ac_1N)}_1 f(N)\left[\frac{s}{C_2}\right]^{-2/\alpha} s^{-1/(ac_1N)}, \qquad\! s \leq \min\{T_{\max}(a),T_{\min}((1+\delta)f(N))\}.
\end{gather*}
Inverting this expression, we f\/ind
\begin{gather*}
T(t) \leq \left[a^{-1}t/f(N) C_1^{-1/(ac_1N)} C_2^{-2/\alpha}\right]^{-\alpha/2 + e_N}, \qquad e_N = \half \frac{\alpha^2}{2 ac_1 N + \alpha},\\
t \geq \max\big\{T^{-1}(T_{\max}(a)), T^{-1}(T_{\min}((1+\delta) f(N)))\big\}.
\end{gather*}

Let us examine this lower bound more carefully. We increase $C_1$ so that $C_1 \leq C_2$, if necessary, so that
\begin{gather*}
T^{-1}(T_{\max}(a)) = a T_{\min}^{-1}(T_{\max}(a)) \leq af(N) \left[ \frac{C_2}{C_1}\right]^{2/\alpha} \leq a f(N),\\
T^{-1}(T_{\min}((1+\delta)f(N)) = (1+\delta)f(N) T_{\max}^{-1}(T_{\min}((1+\delta) f(N))) \\
\hphantom{T^{-1}(T_{\min}((1+\delta)f(N))}{}
\leq a(1+\delta) f(N) \left(1 + \frac{\log C_1/C_2 + \alpha/2 \log (1 + \delta)}{a c_1 N} \right) \\
\hphantom{T^{-1}(T_{\min}((1+\delta)f(N))}{}
\leq a(1+\delta)(1+ CN^{-1}(1 + \alpha \delta)) f(N),
\end{gather*}
where $C$ is a suitable constant. Then using the assumption~\eqref{e:scaling} and $C_2 \geq 1$, $C_1^{\alpha/(2ac_1N)}C_2^{1-2e_N/\alpha}$ is bounded by a constant, say, $C/2 > 0$. This establishes the lemma.
\end{proof}

Let $\Sigma_N$ denote the set of $N \times N$ strictly positive def\/inite complex matrices and recall that the constant $\delta_N$ is def\/ined in Lemma~\ref{l:cond-tail}. The following lemma is applied to control the halting time in terms of the condition number, and the reader may turn to the lemmas that follow to see instances of functions $g$.

\begin{Lemma}\label{l:exp-est}
Let $g\colon [1,\infty) \to \mathbb R$ be continuous and differentiable on $(1,\infty)$. Assume $g$ satisfies $g(1) =0$ and $g'(x) \leq C x^\ell$ for $C> 0$, $\ell \in \mathbb R$ and $x$ sufficiently large. Assume a function $M\colon \Sigma_N \to \mathbb R$ satisfies $M(H) \leq g(\kappa(H))$. Then if $H$ satisfies Condition~{\rm \ref{cond:tails}} and $\alpha/2-e_N-\ell > 1$ there exist constants $C,K > 0$ such that
\begin{gather*}
\mathbb E[M(H)] \leq g(b_N) + C(1+\delta_N)^{-\alpha/K}\frac{b_N^{1+\ell}}{\alpha/2-e_N-\ell-1},
\end{gather*}
for $b_N = a f(N)(1+\delta_N)$.
\end{Lemma}
\begin{proof}
First,
\begin{gather*}
\mathbb E[M(H)] \leq E[g(\kappa(H))],
\end{gather*}
and by integration by parts
\begin{gather*}
\mathbb E[g(\kappa(H))] = \int_1^\infty g(s) \D \mathbb P(\kappa(H) \leq s) = \int_1^\infty g'(s) \mathbb P(\kappa(H) > s) \D s.
\end{gather*}
Using Lemma~\ref{l:cond-tail} for $b_N = a f(N)(1+\delta_N)$ and $N$ suf\/f\/iciently large
\begin{gather*}
\int_1^\infty g'(s) \mathbb P(\kappa(H) > s) \D s \leq \int_1^{b_N} g'(s) \D s + C \left( \frac{b_N}{1 + \delta_N} \right)^{-\alpha/2 + e_N} \int_{b_N}^\infty s^{-\alpha/2+e_N+\ell}\D s\\
\hphantom{\int_1^\infty g'(s) \mathbb P(\kappa(H) > s) \D s }{} \leq g(b_N) + C(1+\delta_N)^{-\alpha/K}\frac{b_N^{1+\ell}}{\alpha/2-e_N-\ell-1}.
\end{gather*}
This last inequality follows from the scaling \eqref{e:scaling}:
 \begin{gather*}
 e_N = \frac{\alpha}{2} \frac{1}{2a c_2 \alpha^{-1} N + 1} = \bigo\big(\alpha N^{\gamma -1}\big),
 \end{gather*}
 and hence $\alpha/2 - e_N = O(\alpha)$. For any f\/ixed $\ell$ there exists a constant $K = K_\ell$
 \begin{gather}\label{e:eN}
 \frac{1}{\alpha/2 - e_N - \ell -1} \leq K/\alpha.
 \end{gather}
\end{proof}
We apply this lemma to the following functions.
\begin{Lemma}\label{l:gs} Let $b$ be a fixed vector then for any $j > 0$
\begin{enumerate}[$(1)$]\itemsep=0pt
\item \textbf{Halting time with the $\ell^2$ norm:} $\tau_{\epsilon}(H,b)^j \leq g(\kappa(H))^j$ where
\begin{gather*}
	g(s) = \frac{\log \sqrt{s} 2\epsilon^{-1}}{\log \left( \frac{\sqrt s + 1}{\sqrt s -1}\right)}.
\end{gather*}
Further, for every $\eta>0$ and $\epsilon >0$ there exists a constant $C_{\epsilon,\eta}$ such that for $s \in [1,\infty)$
\begin{gather*}
 	g(s) \leq \half \sqrt s \log \sqrt{s} 2\epsilon^{-1} \leq C_{\epsilon,\eta} s^{1/2+\eta}, \qquad
 g'(s) \leq \half \frac{\log \sqrt{s} 2\epsilon^{-1}}{\sqrt{s}} \leq C_{\epsilon,\eta} s^{-1/2+\eta}.
\end{gather*}
\item \textbf{Halting time with the weighted norm:} $\tau_{w,\epsilon}(H,b)^j \leq g(\kappa(H))^j$ where
\begin{gather*}
	g(s) = \frac{\log 2\epsilon^{-1}}{\log \left( \frac{\sqrt s + 1}{\sqrt s -1}\right)}.
\end{gather*}
This function $g$ satisfies the following estimates
\begin{gather*}
g(s) \leq \half \sqrt s \log 2\epsilon^{-1}, \quad g'(s) \leq \half \frac{\log 2\epsilon^{-1}}{\sqrt{s}}.
\end{gather*}
\item \textbf{Successive residuals:} For $r_k = r_k(A,b)$, $\left(\frac{\|r_{k+1}\|}{\|r_{k}\|}\right)^j \leq g(\kappa(H))^j$ where
\begin{gather*}
g(s) = \left( \frac{\sqrt{s}-1}{\sqrt{s}+1}\right) \leq 1, \qquad g'(s) \leq s^{-3/2},
\end{gather*}
and $\|\cdot\|$ stands for either $\|\cdot \|_{\ell^2}$ or $\|\cdot\|_A$.
\end{enumerate}
\end{Lemma}
\begin{proof}
All the bounds follow from \eqref{rut} and \eqref{rutA} as explained in the introduction to this section. The estimates on the functions $g(s)$, each of which satisf\/ies $g(1) = 0$, may be obtained by elementary manipulations.
\end{proof}
We can now prove our generalized result.
\begin{Theorem} \label{t:main-gen} Assume a random matrix $H$ satisfies Condition~{\rm \ref{cond:tails}} and $\alpha$ satisfies \eqref{e:scaling}. Then for $b_N = af(N)(1+\delta_N)$, $\delta_N > 0$ and, any vector $b = 1$, the following estimates hold:
\begin{enumerate}[$(1)$]\itemsep=0pt
\item \textbf{Halting time with the $\ell^2$ norm:}
 \begin{gather*}
 \mathbb E\big[\tau_{\epsilon}(H,b)^j\big] \leq \frac{1}{2^j} b_N^{j/2}\big(\log b_N^{1/2}\epsilon^{-1}\big)^j\big(1 + o\big(b_N^{-1/2}(1+\delta_N)^{-\alpha/K}\big)\big), \qquad \text{as}\quad N \to \infty.
 \end{gather*}
\item \textbf{Halting time with the weighted norm:}
 \begin{gather*}
 \mathbb E\big[\tau_{w,\epsilon}(H,b)^j\big] \leq \frac{1}{2^j} b_N^{j/2}\big(\log\epsilon^{-1}\big)^j\big(1 + O\big(b_N^{-1/2} N^{-\gamma}(1+\delta_N)^{-\alpha/K}\big)\big), \qquad \text{as} \quad N \to \infty.
 \end{gather*}
 \item \textbf{Successive residuals:} For $r_k = r_k(H,b)$
 \begin{gather*}
 \mathbb E\left[ \frac{\|r_{k+1}\|^j}{\|r_k\|^j}\right] \leq \left( 1 - \frac{2}{\sqrt{b_N}+1}\right)^j + O\big(b_N^{-1/2}N^{-\gamma}(1+\delta_N)^{-\alpha/K}\big), \qquad \text{as}\quad N \to \infty,
 \end{gather*}
 where $\|\cdot\|$ stands for either $\|\cdot \|_{\ell^2}$ or $\|\cdot\|_{w^{-1}}$.
\end{enumerate}
\end{Theorem}
\begin{proof}
 Before we begin, we recall \eqref{e:eN}.

(1) \textbf{Halting time with the $\ell^2$ norm:} As $H$ satisf\/ies Condition~\ref{cond:tails} we can apply Lemma~\ref{l:exp-est} with the estimates in Lemma~\ref{l:gs}(1) for $j > 0$ and $\eta > 0$. We use that in this case
 \begin{gather*}
 \frac{\D}{\D s} g(s)^j = j g'(s) g(s)^{j-1} \leq jC_{\epsilon,\eta}^j s^{(j-3)/2+ j\eta},
 \end{gather*}
 and hence $\ell = (j-3)/2 + j \eta$ in Lemma~\ref{l:exp-est}. Therefore,
 \begin{gather*}
 \mathbb E\big[\tau_{\epsilon}(H,b)^j\big] \leq \frac{1}{2^j} b_N^{j/2}\big(\log b_N^{1/2}\epsilon^{-1}\big)^j + j C_\ell C_{\epsilon,\eta}^j(1+\delta_N)^{-\alpha/K} b_N^{j/2} b_N^{-1/2 + j \eta}N^{-\gamma}.
 \end{gather*}
 If $b_N = \bigo(N^\zeta)$ for some $\zeta > 0$, we choose $\eta > 0$ such that $j\zeta\eta < \gamma$. Thus,
 \begin{gather*}
 \mathbb E\big[\tau_{\epsilon}(H,b)^j\big] \leq \frac{1}{2^j} b_N^{j/2}\big(\log b_N^{1/2}\epsilon^{-1}\big)^j\big(1 + o\big(b_N^{-1/2}(1+\delta_N)^{-\alpha/K}\big)\big), \qquad \text{as} \quad N \to \infty.
 \end{gather*}

(2) \textbf{Halting time with the weighted norm:} We follow the same calculations as (1) with the estimates in Lemma~\ref{l:gs}(2) for $j > 0$. Here
 \begin{gather*}
 \frac{\D}{\D s} g(s)^j = j g'(s) g(s)^{j-1} \leq \frac{j}{2^{j+1}} s^{(j-3)/2}\big(\log \epsilon^{-1}\big)^j,
 \end{gather*}
 and hence $\ell = (j-3)/2 $ in Lemma~\ref{l:exp-est}. Therefore,
 \begin{gather*}
 \mathbb E\big[\tau_{w,\epsilon}(H,b)^j\big] \leq \frac{1}{2^j} b_N^{j/2}\big(\log \epsilon^{-1}\big)^j + j C_\ell \frac{1}{2^{j+1}} (1+\delta_N)^{-\alpha/K}b_N^{j/2} b_N^{-1/2}N^{-\gamma}.
 \end{gather*}
 Finally,
 \begin{gather*}
 \mathbb E\big[\tau_{w,\epsilon}(H,b)^j\big] \leq \frac{1}{2^j} b_N^{j/2}\big(\log \epsilon^{-1}\big)^j\big(1 + \mathcal O\big(b_N^{-1/2} N^{-\gamma}(1+\delta_N)^{-\alpha/K}\big)\big), \qquad \text{as}\quad N \to \infty.
 \end{gather*}

(3) \textbf{Successive residuals:} We follow the same calculations as (1), (2) with the estimates in Lemma~\ref{l:gs}(3) for $j > 0$. Here
 \begin{gather*}
 \frac{\D}{\D s} g(s)^j = j g'(s) g(s)^{j-1} \leq j s^{-3/2}
 \end{gather*}
 and hence $\ell = -3/2 $ in Lemma~\ref{l:exp-est}. Then with $R_k = \|r_{k+1}\|/\|r_k\|$
 \begin{gather*}
 \mathbb E[R_k^j] \leq \left( 1 - \frac{2}{\sqrt{b_N}+1}\right)^j + j C_\ell b_N^{-1/2}N^{-\gamma}(1+\delta_N)^{-\alpha/K}.
 \end{gather*}
 Therefore
 \begin{gather*}
 \mathbb E[R_k^j] \leq \left( 1 - \frac{2}{\sqrt{b_N}+1}\right)^j + \mathcal O\big(b_N^{-1/2} N^{-\gamma}(1+\delta_N)^{-\alpha/K}\big), \qquad \text{as}\quad N \to \infty.\tag*{\qed}
 \end{gather*}\renewcommand{\qed}{}
\end{proof}

The constant $\delta_N> 0$ is used in the above theorem to make precise the fact that if we integrate the tail of the condition number distribution just beyond $af(N)$ the error term is exponentially small as $\alpha \to \infty$.

\section{The Laguerre unitary ensemble}\label{s:LUE}

The following is well-known and may be found in~\cite[Section~2]{Forrester1993}, for example. This discussion is modif\/ied from \cite[Section~2]{Deift2015}. Let $W = XX^*$ where $X$ is an $N \times (N + \alpha)$ matrix of iid standard complex Gaussian random variables. Recall that the (matrix-valued) random variable $W$ is the Laguerre unitary ensemble (LUE). Then it is known that the eigenvalues $0 \leq \lambda_{\min} = \lambda_1 \leq \lambda_2 \leq \cdots \leq \lambda_N = \lambda_{\max}$ of $W$ have the joint probability density
\begin{gather*}
p_N(\lambda_1, \ldots, \lambda_N) = \frac{1}{C_N^{(\alpha)}} \prod_{j=1}^N \lambda_j^\alpha \E^{-\lambda_j} \prod_{1 \leq j < k \leq N} |\lambda_j-\lambda_k|^2.
\end{gather*}
 Recall that the Laguerre polynomials, $\{\La{j}(x)\}_{j=0}^\infty$, are a family of orthogonal polynomials on $[0,\infty)$, orthogonal with respect to the weight $e^{-x}x^\alpha$. We normalize them as follows~\cite{DLMF}
\begin{gather*}
\La{j}(x)= k_j x^j+ \bigo\big(x^{j-1}\big),\qquad 	k_j = \frac{(-1)^j}{j!}, \\
 \int_0^\infty \La{i}(x)\La{j}(x) e^{-x} x^{\alpha} \D x = \delta_{ij} \frac{\Gamma(j + \alpha + 1)}{j!}.
\end{gather*}
Then the following are orthonormal with respect to Lebesgue measure on $[0,\infty)$,
\begin{gather*}
\psi_j(x) := \left( \frac{j!}{\Gamma(j + \alpha + 1)}\right)^{1/2} e^{-x/2} x^{\alpha/2} \La{j}(x), \qquad \int_0^\infty \psi_j(x) \psi_i(x) \D x = \delta_{ij}.
\end{gather*}
Def\/ine the correlation kernel 
\begin{gather*}
\mathcal K_N (x,y)= \sum_{j=0}^{N-1} \psi_j(x) \psi_j(y), \qquad 0 < x,y < \infty.
\end{gather*}
The kernel $\mathcal{K}_N$ def\/ines a positive, f\/inite-rank and hence trace-class operator on~$L^2([a,b])$. To see that~$\mathcal{K}_N$ is positive, consider $f \in C^\infty((s,t))$ with compact support and note that
\begin{gather*}
\int_s^t \int_s^t \mathcal K_N(x,y) f(x) f^*(y) \D x \D y = \int_s^t \int_s^t \sum_{j=0}^{N-1} \psi_j(x)\psi_j(y) f(x) f^*(y) \D x \D y \\
\qquad{}= \sum_{j=0}^{N-1}\left(\int_s^t \psi_j(x) f(x) \D x\right) \left(\int_s^t \psi_j(x) f^*(x) \D x\right) = \sum_{j=0}^{N-1}\left|\int_s^t \psi_j(x) f(x) \D x \right|^2.
\end{gather*}
The eigenvalues $\lambda_1 \leq \cdots \leq \lambda_N$ may be described in terms of Fredholm determinants of the kernel $\mathcal{K}_N$ \cite{DeiftOrthogonalPolynomials,Forrester1993}. In particular, the statistics of the extreme eigenvalues are recovered from the determinantal formula
\begin{gather*}
\mathbb P\left( \text{no eigenvalues in } [a,b] \right) = \det \big(I - \mathcal K_N|_{L^2([a,b])}\big).
\end{gather*}
By the Christof\/fel--Darboux formula \cite{Szego1959}, we may also write
\begin{gather*}
\mathcal K_N (x,y) = \frac{N!}{\Gamma(N+ \alpha)} \left( \frac{\Gamma(N+ \alpha+1) \Gamma(N+\alpha)}{N!(N-1)!} \right)^{1/2} \left( \frac{\psi_{N-1}(x)\psi_N(y) - \psi_N(x) \psi_{N-1}(y)}{x-y} \right)\\
\hphantom{\mathcal K_N (x,y)}{} = \frac{N!}{\Gamma(N+ \alpha)} e^{-(x+y)/2} x^{\alpha/2}y^{\alpha/2} \frac{\La{N}(y) \La{N-1}(x) - \La{N}(x) \La{N-1}(y)}{x-y}.
\end{gather*}
Thus, questions about the asymptotic behavior of $\mathcal K_N(x,y)$ as $N \goto \infty$ reduce to the study of the large $N$ asymptotics of $\La{N}$ and $\La{N-1}$.

\subsection{Kernel estimates}
We use Fredholm determinants to show that Condition~\ref{cond:tails} holds with appropriate constants when $W = XX^*$ is distributed according to LUE. The main reference for these ideas is \cite{Simon2010}. Let $A\colon L^2([t,\infty)) \to L^2([t,\infty))$ be a positive trace-class operator with kernel $\mathcal K(x,y)$. Assume
\begin{gather*}
 \mathbb P(X \leq t) = \det(I - A_t),
\end{gather*}
then
\begin{gather*}
 \mathbb P(X > t) = | 1 - \det(I - A_t)| = | \det(I) - \det(I - A_t)| \\
\hphantom{\mathbb P(X > t)}{} \leq \left( \int_t^\infty |\mathcal K(x,x)| \D x \right) \exp\left(1 + \int_t^\infty |\mathcal K(x,x)| \D x \right).
 \end{gather*}
In this way we can get estimates on the tail directly from the large $x$ behavior of $\mathcal K(x,x)$. Similar considerations follow if, say, $A_t\colon L^2([0,t]) \to L^2([0,t])$.

As was done in \cite{Deift2015}, we consider the scaled kernel
\begin{gather*}
 \mathcal K^s_N(x,y) = \sum_{j=0}^{N-1} \mathfrak Z_j \pi_j(x)\pi_j(y)x^{\alpha/2} y^{\alpha/2} \E^{-\nu(x+y)/2}, \qquad \mathfrak Z_j^{-1} = \int_0^\infty \pi_j^2(x) x^{\alpha} y^{\alpha/2} \E^{-\nu x}\D x,\\
 \pi_j(x) = \frac{\La{j}(\nu x)}{\nu^j k_j} = x^j + \bigo\big(x^{j-1}\big),
\end{gather*}
so that
\begin{gather*}
 \mathbb P(\lambda_{\max}(W)/\nu \leq t) = \det \big(1- \mathcal K^s_N|_{L^2(t,\infty)}\big),\\
 \mathbb P(\lambda_{\min}(W)/\nu > t) = \det \big(1- \mathcal K^s_N|_{L^2(0,t)}\big).
\end{gather*}

Next, we pull results from~\cite{Deift2015} to estimate the kernel $\mathcal K^s_N(x,y)$ for LUE near the largest and smallest eigenvalue of~$W$. We f\/irst look for the asymptotics of $\mathcal K^s_N(x,y)$ for $(x,y) \approx (1,1)$, called the soft edge. Let $\check x = 1 + x/(2^{2/3} M^{2/3})$ and def\/ine
\begin{gather*}
 \check{\mathcal K}_N(x,y) = \mathcal K^s_N(\check x, \check y) \frac{1}{2^{2/3} M^{2/3}}, \qquad M = N + \half(\alpha + 1).
\end{gather*}
Then from \cite[Proposition~2]{Deift2015}:
 \begin{Proposition}\label{SoftKernelLimit}
As $N \to \infty$ the rescaled kernels converge pointwise,
\begin{gather*}
\check {\mathcal K}_N(x,y) \goto \frac{\Ai(x)\Ai'(y) - \Ai'(x)\Ai(y)}{x-y}, \qquad (x,y) \in \mathbb{R}^2,
\end{gather*}
and the convergence is uniform for $(x,y)$ in a compact subset of $[L,\infty)^2$ for any $L \in \mathbb{R}$. If $x=y$ then the limit is determined by continuity. Further, there exists a positive, piecewise-continuous function $\bar G\colon (L,\infty)^2 \to (0,\infty)$, such that
\begin{gather*}
|\check {\mathcal K}_N( x, y)| \leq \bar G(x,y), \qquad \int_{L}^\infty \int_{L}^\infty \bar G(x,y) \D x\, \D y < \infty, \qquad \int_{L}^\infty \bar G(x,x) \D x < \infty.
\end{gather*}
Furthermore, it suffices to take for a constant $C = C(L) > 0$
\begin{gather*}
 \bar G(x,y) = C\big( \bar g(x)^2 \chi_{|x-y| < 1}(x,y) + \bar g(x) \bar g(y)\big), \qquad \bar g(x) = \begin{cases}1, & \text{if} \ x < 0,\\ \E^{-\frac{1}{3}x^{3/2}}, & \text{if} \ x \geq 0. \end{cases}
\end{gather*}
 \end{Proposition}

For $\mathcal K^s_N(x,y)$ near $(0,0)$ as the scaling of the kernel depends critically on $\gamma$. So, we def\/ine
 \begin{gather*}
 \hat x = \frac{\alpha^2}{\nu^2}\left( 1 + x \left(\frac{2}{\alpha}\right)^{2/3} \right)
 \end{gather*}
and
 \begin{gather*}
 \hat{\mathcal K}_N(x,y) = \mathcal K^s_N(\hat x, \hat y) \frac{2^{2/3}\alpha^{4/3}}{\nu^2}.
\end{gather*}
The next proposition essentially follows directly from \cite[Proposition~1]{Deift2015} and is in fact a little simpler with the scaling chosen here.
\begin{Proposition}\label{HardKernelLimit}
As $N \to \infty$ the rescaled kernels converge pointwise,
\begin{gather*}
\hat {\mathcal K}_N(x,y) \goto \frac{\Ai(-y)\Ai'(-x) - \Ai'(-y)\Ai(-x)}{x-y}, \qquad (x,y) \in \mathbb{R}^2,
\end{gather*}
and the convergence is uniform for $(x,y)$ in any compact subset of $(-\infty,L]^2$ for any $L \in \mathbb{R}$. If $x=y$ then the limit is determined by continuity. Further, there exists a positive, piecewise-continuous function $G\colon (-\infty,L)^2 \to (0,\infty)$, such that
\begin{gather*}
|\hat {\mathcal K}_N( x, y)| \leq G(x,y), \qquad \int_{-\infty}^L \int_{-\infty}^L G(x,y) \D x \D y < \infty, \qquad \int_{-\infty}^L G(x,x) \D x < \infty.
\end{gather*}
Furthermore, it suffices to take for a constant $C = C(L) > 0$
\begin{gather*}
G(x,y) = C (\check g(x) \check h(x) \chi_{|x-y|<1} + \check g(x) \check g(y)),
\end{gather*}
where
\begin{gather*}
 \check g(x) = \begin{cases}
 0, & \text{if} \ -\infty < x \leq -(\alpha/2)^{2/3},\\
 \left[\left(1 + x\left(\dfrac{2}{\alpha} \right)^{2/3} \right) \right]^{d\alpha/4},& \text{if} \ -(\alpha/2)^{2/3} \leq x \leq \big(\ell(d)^2-1\big)(\alpha/2)^{2/3},\\
 \E^{-\frac{1}{6} |x+1|^{3/2}}, & \text{if} \ -b(\alpha/2)^{2/3} < x \leq -1,\\
 1, & \text{otherwise},\end{cases}\\
 \check h(x) = \begin{cases}
 0, & \text{if} \ -\infty < x \leq -(\alpha/2)^{2/3},\\
 \left[\left(1 + x\left(\dfrac{2}{\alpha} \right)^{2/3} \right) \right]^{d\alpha/4-1},& \text{if} \ -(\alpha/2)^{2/3} \leq x \leq \big(\ell(d)^2-1\big)(\alpha/2)^{2/3},\\
 \dfrac{1}{1-\ell(d)^2}\E^{-\frac{1}{6} |x+1|^{3/2}}, & \text{if} \ -b(\alpha/2)^{2/3} < x \leq -1,\\
1, & \text{otherwise}.
 \end{cases}
\end{gather*}
Here $\ell(d)< 1$ satisfies $\ell(d) = 1 + \bigo(d^{2/3})$ as $d \to 0$.
\end{Proposition}

We now brief\/ly describe how the estimates in terms of $\check g$ and $\check h$ arise. The asymptotics of the kernel $\mathcal K_N^s$ is given in terms of Bessel functions, after a change of variables. In the regime $\alpha \to \infty$, the Bessel functions asymptote to Airy functions, as follows~\cite{DLMF}
\begin{gather*}
\Ja(\alpha t) = \left( \frac{ 4 \zeta}{1-t^2} \right)^{1/4} \alpha^{-1/3} \big( \Ai\big(\alpha^{2/3} \zeta\big) + \Ai'\big(\alpha^{2/3} \zeta\big) \bigo \big(\alpha^{-4/3}\big)\big), \qquad t > 0,\\
{\Ja}'(\alpha t) = - \frac{2}{t} \left( \frac{ 4 \zeta}{1-t^2} \right)^{-1/4} \alpha^{-2/3} \big( \Ai'\big(\alpha^{2/3} \zeta\big) + \Ai\big(\alpha^{2/3} \zeta\big)\bigo \big(\alpha^{-2/3}\big)\big), \qquad t > 0,\\
\frac{2}{3} \zeta^{3/2} = \int_t^1 \frac{\sqrt{1-s^2}}{s} \D s, \qquad 0 < t \leq 1,\\
\frac{2}{3} (-\zeta)^{3/2} = \int_1^t \frac{\sqrt{s^2-1}}{s} \D s, \qquad t > 1.
\end{gather*}
This expansion is uniform for $t \in (0,\infty)$. Assume $z \in (0,\delta')$ where $\delta' < \delta < 1/2$ and $t$ is given in terms of $z$ by $t = -\I\frac{\nu}{4 \alpha} \phi_\rightarrow^+(z)$. The following are from \cite{Deift2015}
 \begin{gather*}
-\I \phi_\rightarrow^+(z) = 2 \int_0^z \sqrt{\frac{1-s}{s}} \D s \leq 4 \sqrt{z},\qquad 0 \leq z \leq \delta',\\
-\I \phi_\rightarrow^+(z) \geq 4 \sqrt{z} |1-\delta|^{1/2} \geq 2 \sqrt{z},\qquad 0 \leq z \leq \delta',\\
\frac{2}{3} \zeta^{3/2} = \int_t^1 \sqrt{1-s} \frac{\sqrt{1+s}}{s} \D s \geq \frac{4}{3} (1-t)^{3/2},\qquad 0 \leq t \leq 1,\quad
\zeta \geq 2^{1/3} (1-t),\quad 0 \leq t \leq 1,\\
\frac{2}{3} \zeta^{3/2} \geq\int_t^1 \frac{\sqrt{1-s}}{s} \D s = -\log 2t + \log 2 + 2 \log \big(1 + \sqrt{1-t}\big) - 2 \sqrt{1-t},\qquad 0 \leq t \leq 1,\\
\zeta \geq \left(\frac{3}{2}\right)^{2/3} |\log 2t|^{2/3} ,\qquad 0 \leq t \leq 1/2.
 \end{gather*}
A subtle issue is the validity of the last bound. We see that $- \I \phi^+_\rightarrow(z) \leq 4\sqrt{z}$, and so $t(z) \leq \frac{\nu}{\alpha}\sqrt{z}$ and $t(\check z) \leq (1 + z (2/\alpha)^{2/3})^{1/2}$. Then considering Lemma~\ref{l:exp-est}, we see that the dominant contribution arises from the interval over which the estimate $\mathbb P(\kappa(H) > s) \leq 1$ is used. Thus, we try to extend the validity of a lower bound on~$\zeta$ to $t \in [0,1]$. It follows that
 \begin{gather*}
\zeta \geq d^{2/3} |\log t|^{2/3} ,\qquad 0 \leq t \leq \ell(d) < 1.
 \end{gather*}
 Note that $\ell(d) \neq 1$ as $\zeta'(t)$ has a bounded derivative at $t = 0$ and the right-hand side does not. But as $d \to 0$, $\ell(d) \to 1$. A quick calculation, using an expansion near $t = 1$ gives
 \begin{gather*}
 \zeta(t) = a (1-\ell(d))(1 + \bigo(1-\ell(d))) = d^{2/3} (-\log \ell(d))^{2/3},
 \end{gather*}
 and this implies:
 \begin{gather}\label{e:ell}
 \begin{split}
 & [1-\ell(d)]/d^{2/3} \to 0,\qquad \ell(d) = 1 + o\big(d^{2/3}\big),\\
 & [1-\ell(d)]/d^{2/3+\epsilon} \to \infty, \qquad \epsilon >0.
 \end{split}
 \end{gather}
 Then, following \cite[equations~(C.3) and (C.4)]{Deift2015},
 \begin{gather*}
 \E^{-\alpha \frac{1}{3} |\zeta|^{3/2}} \leq t^{d\alpha/2},\qquad 0 \leq t \leq \ell(d),\\
 \E^{-\alpha \frac{1}{3} |\zeta|^{3/2}} \leq \left( \frac{\nu}{\alpha} \sqrt{z} \right)^{d\alpha/2}, \qquad 0 \leq z \leq \ell(d)^2 \frac{\alpha^2}{\nu^2}.
 \end{gather*}
 This last inequality implies that $t \leq \ell(d)$. Then
 \begin{gather*}
 0 \leq \check z \leq \ell(d)^2 \frac{\alpha^2}{\nu^2},\\
 0 \leq 1 + z(2/\alpha)^{2/3} \leq \ell(d)^2,\\
 -(\alpha/2)^{2/3} \leq z \leq (\ell(d)^2-1)(\alpha/2)^{2/3}.
 \end{gather*}
 These estimates can then be plugged into \cite[Lemma~C.2]{Deift2015} to get the estimates in Lemma~\ref{HardKernelLimit}.

\subsection{Tail bounds}\label{s:tail}

 It follows that
 \begin{gather*}
 \mathbb P(\lambda_{\max}(W)/\nu > t) \leq \left(\int_{t}^\infty |\mathcal K^s_N(x,x)|\D x \right)\exp\left(1 + \int_{t}^\infty |\mathcal K^s_N(x,x)|\D x \right).
 \end{gather*}
 So, we estimate for $t \geq 1$ and $C \geq 1$
 \begin{gather*}
 \int_{t}^\infty |\mathcal K^s_{N}(x,x)| \D x \leq 2^{5/3} M^{2/3} C \int_{t}^\infty \bar g\big( 2^{2/3}M^{2/3}(x-1)\big)^2 \D x\\
 \hphantom{ \int_{t}^\infty |\mathcal K^s_{N}(x,x)| \D x}{} = 2C\int_{2^{2/3} M^{2/3}(t-1)}^\infty \bar g^2(s) \D s \leq 2C \E^{-\frac{4}{3} M(t-1)}.
 \end{gather*}
 So, for a new constant $C$
 \begin{gather*}
 T_{\max}(t)=\mathbb P(\lambda_{\max}(W)/\nu > t) \leq C\E^{-\frac{4}{3} M(t-1)}.
 \end{gather*}
 The more delicate estimate is to consider $T_{\min}$:
 \begin{gather*}
 T_{\min}(t) = \mathbb P\big(\nu \lambda_{\min}^{-1}(W) > t\big) = \mathbb P\big(\nu^{-1} \lambda_{\min}(W) < t^{-1}\big) = 1 - \mathbb P\big(\nu^{-1} \lambda_{\min}(W) \geq t^{-1}\big)\\
 \hphantom{T_{\min}(t)}{} \leq \left(\int_{0}^{t^{-1}} |\mathcal K^s_N(x,x)|\D x \right)\exp\left(1 + \int_{0}^{t^{-1}} |\mathcal K^s_N(x,x)|\D x \right).
 \end{gather*}
 We use Proposition~\ref{HardKernelLimit} and invert the scaling $\hat x$. If $\hat x$ lies in $[0,\alpha^2/\nu^2]$ then $x \in (-\infty,1]$. We note that $\check g(x)^2 \leq \check g(x)\check h(x)$ so we only need to estimate
 \begin{gather*}
 \check g\left( \left(\frac{\alpha}{2}\right)^{2/3} \left( \frac{\nu^2}{\alpha^2} x - 1 \right) \right)\check h\left( \left(\frac{\alpha}{2}\right)^{2/3} \left( \frac{\nu^2}{\alpha^2} x - 1 \right) \right) \leq \left[ \frac{\nu^2}{\alpha^2} x \right]^{d\alpha/2} x^{-1},
 \end{gather*}
 for $0 \leq x \leq \ell(d)^2\frac{\alpha^2}{\nu^2}$. For $t \leq \ell(d)^{-2}\frac{\nu^2}{\alpha^2}$ we just use $T_{\min}(t) \leq 1$. For $t \geq \ell(d)^{-2}\frac{\nu^2}{\alpha^2}$ and $\alpha > 0$
 \begin{gather*}
 \int_0^{t^{-1}} |\mathcal K^s_N(x,x)| \D x \leq 2C \frac{\nu^2}{\alpha^2} \alpha^{2/3} \int_{0}^{t^{-1}} \left[\frac{\nu^2}{\alpha^2} x \right]^{d\alpha/2} x^{-1} \D x\\
 \hphantom{ \int_0^{t^{-1}} |\mathcal K^s_N(x,x)| \D x}{}
 = 2C \frac{\alpha^{2/3}}{d\alpha/2}\left[\frac{\nu^2}{\alpha^2}\right]^{d\alpha/2+1} t^{-d\alpha/2} \leq C d^{-1}\alpha^{-1/3} \left[ \frac{\nu^2}{\alpha^2}\right]^{d\alpha/2+1} t^{-d\alpha/2}.
 \end{gather*}
 It then follows that for a constant $C > 0$
 \begin{gather*}
 T_{\min}(t) \leq C d^{-1}\alpha^{-1/3} \left[\frac{\nu^2}{\alpha^2}\right]^{d\alpha/2+1} t^{-\alpha/2}, \qquad t \geq \ell(d)^{-2} \frac{\nu^2}{\alpha^2}.
 \end{gather*}
 We arrive at the following.
 \begin{Lemma}\label{l:LUEest}
 If $W = XX^*$ where $X$ is an $N\times (N + \alpha)$ matrix of iid standard complex normal random variables, $\alpha = \lfloor \sqrt{4c N} \rfloor$ and $\nu = 4N + 2 \alpha + 2$ then
 \begin{gather*}
 	\mathbb P\big(\lambda_{\max}\big(\nu^{-1}W\big) > t\big) \leq C\E^{4/3 M(t-1)},\qquad
 \mathbb P\big(\lambda_{\min}^{-1}\big(\nu^{-1}W\big) >t\big) \leq C \left[f(N)/t\right]^{d\alpha/2}, \\
 f(N) = \big[ d^{-1} \alpha^{-1/3}\big]^{2/(d\alpha)} \left[ \frac{\nu^2}{\alpha^2} \right]^{1 + 2/(d\alpha)},\qquad
 t \geq \ell(d)^{-2} \frac{\nu^2}{\alpha^2}.
 \end{gather*}
 \end{Lemma}
 The following lemma is a generalization of this result, it essentially follows from the analysis in~\cite{Deift2015}, by allowing $c \to 0$ in~\eqref{e:scaling} at some rate in $N$ as the estimates there are uniform for~$c$ bounded. We do not present a proof here as this will be included in a forthcoming work.
 \begin{Lemma}\label{l:tail}
 If $W = XX^*$ where $X$ is an $N\times (N + \alpha)$ matrix of iid standard complex normal random variables, $\alpha = \lfloor \sqrt{4c} N^\gamma \rfloor$, $0< \gamma \leq 1/2$ and $\nu = 4N + 2 \alpha + 2$ then
 \begin{gather*}
 	\mathbb P\big(\lambda_{\max}\big(\nu^{-1}W\big) > t\big) \leq C\E^{-4/3 M(t-1)},\qquad
 \mathbb P\big(\lambda_{\min}^{-1}\big(\nu^{-1}W\big) >t\big) \leq C [f(N)/t]^{d\alpha/2},\\
 f(N) = \big[ d^{-1} \alpha^{-1/3}\big]^{2/(d\alpha)} \left[ \frac{\nu^2}{\alpha^2} \right]^{1 + 2/(d\alpha)},\qquad
 t \geq \ell(d)^{-2} \frac{\nu^2}{\alpha^2}.
 \end{gather*}
 \end{Lemma}
 It is conjectured that these same estimates hold for $1/2 < \gamma \leq 1$ also but this does not follow immediately from the work in~\cite{Deift2015}.

 \begin{proof}[Proof of Theorem~\ref{t:smooth}]
It follows that
\begin{gather*}
\kappa\big(A + \sigma^2 H\big) = \kappa\big(\sigma^{-2}A + H\big).
\end{gather*}
Then
\begin{gather*}
\mathbb P\big(\lambda_{\max}\big(\sigma^{-2}A + H\big) > t\big) \leq \mathbb P\big(\lambda_{\max}(H) > t - \sigma^{-2}\big),\\
\mathbb P\big(\lambda_{\min}^{-1}\big(\sigma^{-2}A + H\big) > t\big) \leq \mathbb P\big(\lambda_{\min}^{-1}(H) > t\big).
\end{gather*}
Since we may choose $d$ as needed, we assume that $\alpha d \sim N^\lambda \to \infty$. We then have from Lemma~\ref{l:tail}, with a possibly new constant~$C$,
\begin{gather*}
 	\mathbb P(\lambda_{\max}(H) > t) \leq C\E^{-4/3 M(x-1)},\\
 \mathbb P\big(\lambda_{\min}^{-1}(H) >t\big) \leq C [f(N)/t]^{d\alpha/2}, \qquad f(N) = \left[ \frac{\nu^2}{\alpha^2} \right].
\end{gather*}
This follows from the fact that $(N^q)^{N^{-\lambda}} \to 1$ for any value of $q$. We def\/ine $\delta$ by $1+\delta = \ell(d)^{-2}$ and then the matrix $A + \sigma^2 H$ satisf\/ies Condition~\ref{cond:tails} with $c_1 = 4/3$, $a = 1 + \sigma^{-2}$, $\alpha \to d\alpha$ and~$f$ and~$\delta$ as def\/ined here. Dif\/ferent values of $\lambda$ can be used to create dif\/ferent estimates. But for simplicity, we take $\lambda = \gamma/2$ or $d = N^{-\gamma/2}$. Then by~\eqref{e:ell}, for $\epsilon$ small, $\delta = 1 - \ell(d)^{-2} \geq N^{-\gamma/3-\epsilon}$ if~$N$ is suf\/f\/iciently large and
\begin{gather*}
 \log (1+\delta)^{N^{\gamma/2}} \to \infty,
\end{gather*}
with some power of~$N$. Therefore $(1+\delta)^{-d\alpha/K}$ tends to zero faster than any power of $N$ if $K > 0$ is f\/ixed. We now establish each estimate by appealing to Theorem~\ref{t:main-gen}.
\begin{enumerate}[(1)]\itemsep=0pt
\item \textbf{Halting time with the $\ell^2$ norm:} Using $b_N = (1+ \sigma^{-2})f(N)(1+\delta_N)$ it follows directly that
 \begin{gather*}
 \mathbb E\big[\tau_{\epsilon}\big(A + \sigma^2 X,b\big)^j\big] \leq \frac{1}{2^j} b_N^{j/2} \big(\log b_N^{1/2} 2\epsilon^{-1}\big)^j + \bigo\big(N^{-k}\big), \qquad \text{for all} \quad k>0.
 \end{gather*}
 \item \textbf{Halting time with the weighted norm:} Again, using $b_N = (1+ \sigma^{-2})f(N)(1+\delta_N)$ it follows directly that
 \begin{gather*}
 \mathbb E\big[\tau_{w,\epsilon}\big(A + \sigma^2 X,b\big)^j\big] \leq \frac{1}{2^j} b_N^{j/2} \big(\log 2\epsilon^{-1}\big)^j + \bigo\big(N^{-k}\big), \qquad \text{for all} \quad k > 0,
 \end{gather*}
\item \textbf{Successive residuals:} Similarly, it follows directly that
 \begin{gather*}
 \mathbb E\left[ \frac{\|r_{k+1}\|^j}{\|r_k\|^j}\right] \leq \left( 1 - \frac{2}{\sqrt{b_N}+1}\right)^j + \bigo\big(N^{-k}\big), \qquad \text{for all} \quad k > 0.
 \end{gather*}
\end{enumerate}
By equation \eqref{e:scaling}, $b_N = c^{-1}(1+ \sigma^{-2}) 4N^{2-2\gamma}(1 + \bigo(N^{-\gamma/3}))$, and the result follows.
\end{proof}

\subsection*{Acknowledgments}

This work was supported in part by grants NSF-DMS-1411278 (GM) and NSF-DMS-1303018 (TT). The authors thank Anne Greenbaum and Zden\v{e}k Strako\v{s} for useful conversations, Folkmar Bornemann for suggesting that we consider the framework of smoothed analysis and the anonymous referees for suggesting additional numerical experiments.

\vspace{-2mm}

\pdfbookmark[1]{References}{ref}
\LastPageEnding

\end{document}